\documentclass[a4paper,11pt]{article}
\pdfoutput=1

\usepackage{mymacros,oitp}
\usepackage{asymptote}

\newcommand{\dge}{\rotatebox[origin=c]{45}{$\ge$}}
\newcommand{\uge}{\rotatebox[origin=c]{315}{$\ge$}}

\newcommand{\dndots}{\rotatebox[origin=c]{315}{$\cdots$}}

\newcommand{\gcbox}[1]{\hbox to1.65em{$\hfill {#1} \hfill$}}
\newcommand{\ldbox}[1]{\hbox to2.2em{$\hfill {#1} \hfill$}}

\newcommand{\GITl}{\backslash \!\! \backslash}
\newcommand{\GIT}{/\!\!/}

\title{Potential functions on Grassmannians of planes\\
and cluster transformations}
\author{Yuichi Nohara, Kazushi Ueda}
\date{}
\begin{document}

\maketitle

\begin{abstract}
With a triangulation of a planar polygon with $n$ sides,
one can associate an integrable system on the Grassmannian of 2-planes
in an $n$-space.
In this paper,
we show that the potential functions
of Lagrangian torus fibers of the integrable systems
associated with different triangulations
glue together
by cluster transformations.
We also prove that the cluster transformations coincide with
the wall-crossing formula 
in Lagrangian intersection Floer theory.
\end{abstract}

\section{Introduction}

Quantum cohomologies of Grassmannians
give quantum deformations of the classical Schubert calculus.
It is a fascinating subject,
which is related to many branch of mathematics
such as
moduli of vector bundles on a Riemann surface
\cite{MR1358625},
total positivity
\cite{MR1869115},
and eigenvalue problems
\cite{MR2008438}
to name a few.

Mirror symmetry is a powerful tool
to study quantum cohomologies of symplectic manifolds.
The mirror of a Fano manifold
is a \emph{Landau--Ginzburg model},
i.e.,
a pair $(\Xv, W)$
of an analytic space $\Xv$
and an analytic function
$W \colon \Xv \to \bA^1$
called the \emph{Landau--Ginzburg potential}.
%
Landau--Ginzburg mirrors of flag varieties
are introduced in \cite{MR2397456},
where $\Xv$ are the complements
of anti-canonical divisors
in the flag manifolds
associated with the Langlands dual groups,
and $W$ are regular functions.
In the type $A$ cases,
the restrictions of $W$ to certain open subvarieties 
give mirrors introduced earlier in
\cite{MR1439892,MR1756568}.
For the Grassmannian $\Gr(k,n) = \Gr(k, \bC^n)$ 
of $k$-dimensional subspaces in $\bC^n$, 
Marsh and Rietsch \cite{1307.1085} give a description of
the Landau--Ginzburg mirror
\begin{equation}
  W \colon \Xv = \Gr(n-k, (\bC^n)^*) \setminus D 
  \longrightarrow \bA^1
\end{equation}
in terms of Pl\"ucker coordinates
on the dual Grassmannian $\Gr(n-k, (\bC^n)^*)$
of $\Gr(k,\bC^n)$.

With a Lagrangian submanifold of a symplectic manifold,
one can associate the \emph{potential function},
which is a Floer-theoretic quantity
obtained as the generating function
of numbers of pseudo-holomorphic disks
bounded by Lagrangian submanifolds
\cite{MR2553465}.
In the case of toric manifolds,
the potential functions of Lagrangian orbits of the torus action
can be identified with the Landau--Ginzburg potentials
of the mirrors.

In contrast to the toric cases
where the toric moment maps
give canonical Lagrangian torus fibrations,
there are a priori no preferred Lagrangian torus fibrations
on flag manifolds.
In the case of the Grassmannian $\Gr(2,n)$ of 2-planes,
with any triangulation $\Gamma$ of 
a convex polygon with $n$ sides,
one can associate a completely integrable system 
\begin{equation}
  \Psi_{\Gamma} \colon \Gr(2,n) \longrightarrow 
  \bR^{2n-4}
\end{equation}
whose image 
$\Delta_{\Gamma} = \Psi_{\Gamma}(\Gr(2,n))$ is a convex polytope.
Note that the number of ways
to triangulate a convex $n$-gon is 
given by the Catalan number
$C_{n-2} = \frac{1}{n-1}\binom{2n-4}{n-2}$.
The potential function of Lagrangian torus fibers
of the integrable system $\Psi_{\Gamma}$
is computed in \cite[Theorem 1.6]{MR3211821}.
It is written as a Laurent polynomial
$W_\Gamma$,
which gives a regular function
on a torus $(\bG_m)^{2n-4}$.

If two triangulations $\Gamma$ and $\Gamma'$
are related by a Whitehead move
(see Figure \ref{fg:flip1}),
the corresponding potential functions
$W_{\Gamma}$ and $W_{\Gamma'}$ are related
by a subtraction-free birational change of variables
of the form
\begin{align}
  W_{\Gamma'}(\ldots, y', y_1, y_2, y_3, y_4, \ldots)
 = W_{\Gamma}(\ldots, y, y_1, y_2, y_3, y_4, \ldots)
\end{align}
where
\begin{gather} \label{eq:coord_change}
 y' = \frac{1}{y} \cdot \frac{y_1 y_2 y_3 y_4}{y_1 y_3 + y_2 y_4}.
\end{gather}
Moreover, the tropicalization of this coordinate change 
gives a piecewise-linear transformation on $\bR^{2n-4}$
which maps $\Delta_{\Gamma}$ into $\Delta_{\Gamma'}$.

In what follows we identify the dual Grassmannian 
$\Gr(n-2, (\bC^n)^*)$ with $\Gr(2,n)$ in a canonical way,
and write the Pl\"ucker coordinates as $p_{ij}$
($1 \le i < j \le n$).
The first main result in this paper is the following:

\begin{theorem} \label{th:main}
For any triangulation $\Gamma$,
there is an open embedding
\begin{align}
 \iota_\Gamma \colon \lb \bGm \rb^{2n-4} \hookrightarrow \Xv
\end{align}
such that the restriction of the Landau--Ginzburg potential
coincides with the potential function;
$
 \iota_\Gamma^* W = W_\Gamma.
$
The change of variables
\eqref{eq:coord_change}
can be identified with the Pl\"ucker relation
\begin{align} \label{eq:pluecker}
 p_{ik} p_{jl} = p_{ij} p_{kl} + p_{il} p_{jk}
\end{align}
by a suitable choice of a coordinate on $\lb \bGm \rb^{2n-4}$.
\end{theorem}

In other words,
the potential functions for different triangulations glue together
to form an open dense subset of Marsh--Rietsch's mirror.
The Pl\"ucker relation \pref{eq:pluecker}
is a prototypical example
of a \emph{cluster transformation}
in the theory of cluster algebras
\cite{MR1887642}.

\begin{remark}
Rietsch and Williams \cite{1507.07817} also study 
the relation between piecewise-linear transformations 
for ``moment polytopes'' $\Delta_{\Gamma}$
and cluster transformations \pref{eq:pluecker}
from a slightly different view point, where $\Delta_{\Gamma}$
are regarded as Newton--Okounkov bodies.
\end{remark}

For a pair $\Gamma$ and $\Gamma'$
of triangulations related by a Whitehead move,
one can construct a one-parameter family $\Psi_t$ 
($0 \le t \le 1$) of completely integrable systems 
on $\Gr(2,n)$ such that 
$\Psi_0 = \Psi_{\Gamma}$ and $\Psi_1 = \Psi_{\Gamma'}$
(up to coordinate changes on the base spaces). 
For $t \ne 0, 1$, the integrable system $\Psi_t$ has singular fibers
over a codimension two subset in the interior of the base space
$B_t= \Psi_t(\Gr(2,n))$.
The presence of singular fibers leads to a codimension one \emph{wall} in $B_t$
which divides $B_t$ into two \emph{chambers}.
The \emph{SYZ mirror}
in the sense of \cite[Definition 1.2]{MR3502098}
of $\Gr(2,n)$
with respect to this Lagrangian torus fibration
is given by gluing
(open subsets of)
Landau--Ginzburg models
$\lb \lb \bGm \rb^{2n-4}, W_\Gamma \rb$
and $\lb \lb \bGm \rb^{2n-4}, W_{\Gamma'} \rb$
(and then completing).
Each of these Landau--Ginzburg models
comes from the moduli space of objects
of the Fukaya category
supported by Lagrangian torus fibers
above each chamber,
and the gluing is given by
a \emph{wall-crossing formula}
obtained by counting
pseudo-holomorphic disks of Maslov index zero.
The second main result in this paper is the following:

\begin{theorem} \label{th:main2}
For $0 < t <1$, the wall-crossing formula
in the construction of the SYZ mirror of $\Gr(2,n)$
with respect to the Lagrangian torus fibration $\Psi_t$
is given by the coordinate change \pref{eq:coord_change}.
\end{theorem}

\pref{th:main2}
is proved by reduction
to the case of $\Gr(2,4)$
by a degeneration argument,
which is then handled directly
along the lines of \cite{MR2386535, MR2537081}.

It is suggested in \cite{MR2386535, MR2537081}
that the mirror of a Fano manifold is obtained by
first taking a special Lagrangian torus fibration
on the complement of an anti-canonical divisor,
and then equipping its Strominger--Yau--Zaslow mirror
with the potential function of the fiber.
The integrable system $\Psi_t$
does not restrict to a Lagrangian torus fibration
on the complement of an anti-canonical divisor, and
it is an interesting problem to find a Lagrangian torus fibration
on the complement of an anti-canonical divisor,
which allows one to fit Rietsch's mirror
into this framework.
Another interesting question
is whether there are other mirrors
associated with other Lagrangian torus fibrations
on the complement of other anti-canonical divisors.


This paper is organized as follows.
After fixing notation for triangulations of convex polygons
in Section 2,
we give  in Section 3
the construction of completely integrable systems
on $\Gr(2,n)$.
In Section 4 we recall toric degenerations of $\Gr(2,n)$
associated with triangulations of a convex $n$-gon,
which enables us to compute potential functions of
Lagrangian torus fibers of $\Psi_{\Gamma}$.
In Section 5 we show that the potential functions 
for different triangulations
are related by the coordinate change  \pref{eq:coord_change}.
Section 6 is a quick review of cluster algebras.
Theorem \ref{th:main} is proved in Section 7.
In Section 8 we recall the wall-crossing formula 
given by Auroux in
\cite{MR2386535, MR2537081},
which is enough for our purpose since the integrable system
$\Psi_t$ has only one wall.
In Section 9 we prove Theorem \ref{th:main2} in the case 
of $\Gr(2,4)$.
The proof for general $\Gr(2,n)$ is given in Section 10.

\textit{Acknowledgment}:
We thank Yank\i\ Lekili
for collaboration
at an early stage of this work;
it is originally conceived as a joint project with him.
We thank River Chiang
for organizing a workshop in Tainan
in July 2014,
where this project has been initiated.
We also thank the anonymous referee
for reading the manuscript carefully,
pointing out mistakes, and
suggesting a number of improvements.
Y.~N. is supported
by  Grant-in-Aid for Scientific Research (15K04847).
K.~U. is supported
by Grant-in-Aid for Scientific Research
 (24740043, 15KT0105, 16K13743, 16H03930).


\section{Triangulations}

Fix an integer $n$ greater than 2,
and let $P$ be a convex planar polygon with $n$ sides 
called the \emph{reference polygon}.
We order the vertices of $P$
in such a way that respects the natural cyclic order
on the boundary of $P$,
and define the side vectors $e_i \in \bR^2$ for $i=1, \ldots, n$
as the difference between the $i$-th vertex
and the $(i+1)$-st vertex.
Setting $I(i,j) = \{i, i+1, i+2, \dots , j-1 \}$,
we can write the diagonal
connecting the $i$-th vertex and the $j$-th vertex as
\begin{align}
 d_{ij} = \sum_{k \in I(i,j)} e_k.
\end{align}
Take a subdivision $\Gamma$ of $P$ given by 
a set of diagonals which are pairwise non-crossing
in the interior of $P$.
Note that the non-crossing condition for diagonals $d_{ij}$, $d_{kl}$
 is equivalent to 
\begin{equation}
  I(i,j) \subset I(k,l) \  \text{or} \ 
  I(k,l) \subset I(i,j) \  \text{or} \ 
  I(i,j) \cap I(k,l) = \emptyset.
  \label{eq:non-crossing}
\end{equation}
We consider the dual graph of the subdivision $\Gamma$,
which is a tree with $n$ leaves.
Let $\edg(i,j)$ denote an edge in the graph 
intersecting a diagonal $d_{ij}$ or a side $e_i = d_{i,i+1}$ of $P$
connecting the $i$-th and $j$-th vertices,
where we assume $\edg(n, n+1) = \edg(1,n)$.
\begin{figure}[h]
  \centering
  \includegraphics[bb=0 0 100 95]{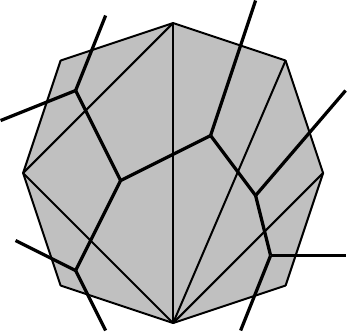}
  \caption{A triangulation of a convex polygon and  its dual graph.}
  \label{fg:trianglation1}
\end{figure}
In what follows we regard $\Gamma$ as the set of edges
in the dual graph by abuse of notation,
and let
\begin{align}
  \Int \Gamma &= \{ \edg(i,j) \in \Gamma \mid |i-j| \ge 2 \}, \\
  \partial \Gamma &= \{ \edg(i, i+1) \mid i=1, \dots, n \}
\end{align}
be the sets of interior edges and 
leaves of $\Gamma$, respectively.
We also consider a ``pruned'' tree obtained from $\Gamma$ 
by removing the $n$-th leaf $\edg(1,n)$, and let
\begin{equation}
  \Prn \Gamma = \Gamma \setminus \{ \edg(1,n) \}
\end{equation}
be the set of its edges.
If $\Gamma$ is a triangulation of $P$,
which is given by $n-3$ diagonals, 
then we have
\begin{align}
  \# \Gamma &= 2n-3, \\
  \# \partial \Gamma &= n = \dim \bT_{U(n)}, \\
  \# \Prn \Gamma &= 2n-4 = \dim_{\bC} \Gr(2,n), \\
  \# \Int \Gamma &= n-3
    = \dim_{\bC} \Gr(2,n) \GIT \bT_{U(n)},
\end{align}
where $\bT_{U(n)} \subset U(n)$ is a maximal torus
consisting of diagonal matrices.


\section{Integrable systems on $\Gr(2,n)$} 

Fix a constant $\lambda >0$,
and identify the Grassmannian $\Gr (2,n)$ of 2-planes in $\bC^n$ 
with 
the adjoint orbit 
\begin{equation}
  \scO_{\lambda} = \{ x \in \sqrt{-1} \fraku(n)
  \mid \text{ eigenvalues of $x$ are $\lambda, \lambda, 0, \dots, 0$} \}
\end{equation}
of a diagonal matrix 
$\diag (\lambda, \lambda, 0, \dots ,0)$
in the space $\sqrt{-1}\fraku (n)$ of Hermitian matrices.
For each $1 \le i < j \le n+1$, 
we consider the adjoint action on $\scO_{\lambda}$ of the subgroup
\begin{equation}
  G(i,j) = \begin{pmatrix}
  \bsone_{i-1} \\
  & U(j-i) \\
  & & \bsone_{n-j+1}
  \end{pmatrix}
  \cong U(j-i)
\end{equation}
of $U(n)$.
Its moment map is given by
\begin{equation}
  \mu_{G(i,j)} \colon \scO_{\lambda} \to \sqrt{-1} \fraku (j-i),
  \quad
  x = (x_{kl}) \mapsto 
  \mu_{G(i,j)}(x) = (x_{kl})_{ k,l \in I(i, j)},
\end{equation}
where we identify the dual space of 
the Lie algebra $\Lie G(i,j) \cong \fraku(j-i)$ of $G(i,j)$ with 
$\sqrt{-1} \fraku(j-i)$ by an invariant inner product.
Since
\begin{equation}
  \rank \mu_{G(i,j)}(x) \le \rank x \le 2,
  \quad x \in \scO_{\lambda}
\end{equation}
for each $(i,j)$ with  $|i-j| \ge 2$,
each Hermitian matrix $\mu_{G(i,j)}(x)$ has 
at most two nonzero eigenvalues
$\lambda^{(i,j)}_1(x) \ge \lambda^{(i,j)}_2(x) \ge 0$.
For each $1 \le i < j \le n$,
we define a function $\psi_{ij}$
on $\scO_{\lambda}$ by
\begin{equation}
  \psi_{ij}(x) = \begin{cases}
      \lambda^{(i,j)}_1(x) , & \text{if $|i-j| \ge 2$},\\
      \mu_{G(i, i+1)}(x) = x_{ii}, \quad &  j=i+1.
  \end{cases} 
\end{equation}
Note that the the moment map
$\mu_{\bT_{U(n)}} \colon \scO_{\lambda} \to \bR^n$ 
of the action of the maximal torus 
$\bT_{U(n)} = \prod_{i=1}^{n} G(i,i+1)$ is given by
\begin{equation}
  \mu_{\bT_{U(n)}} = (\psi_{12}, \psi_{23}, \psi_{34}, \dots, \psi_{n,n+1}),
\end{equation}
and hence  $\{\psi_{i, i+1}\}_{1 \le i \le n}$ satisfies one relation
\begin{equation}
  \psi_{12}(x) + \psi_{23}(x) + \dots + \psi_{n-1,n}(x) + \psi_{n,n+1}(x)
  = \tr x = 2 \lambda.
\end{equation}
In general, 
for any $n \times n$ Hermitian matrix $x = (x_{ij})$,
the mini-max principle implies that 
the eigenvalues $\lambda_1 \ge \dots \ge \lambda_n$ of $x$ and
those $\mu_1 \ge \dots \ge \mu_{n-1}$ of the submatrix
$(x_{ij})_{1 \le i, j \le n-1}$ satisfy
\begin{equation}
  \lambda_1 \ge \mu_1 \ge \lambda_2 \ge \mu_2 \ge \lambda_3 \ge
  \dots \ge \lambda_{n-1} \ge \mu_{n-1} \ge \lambda_n.
\end{equation}
In our situation, each $x \in \scO_{\lambda}$ has eigenvalues 
$\lambda, \lambda, 0, \dots, 0$, and hence
the largest eigenvalue $\psi_{1n}(x) = \lambda^{(1,n)}_1(x) $
of $\mu_{G(1,n)}(x) = (x_{ij})_{1 \le i,j \le n-1}$ is constant:
\begin{equation}
\psi_{1n} = \lambda.
\end{equation}
For a triangulation $\Gamma$ of the reference $n$-gon $P$,
the non-crossing condition \eqref{eq:non-crossing}
implies that
$G(i,j) \subset G(k,l)$ or $G(k,l) \subset G(i,j)$
or the actions of $G(i,j)$ and $G(k,l)$ on $\scO_{\lambda}$ 
commute
for each pair $\edg(i,j), \edg(k,l) \in \Int \Gamma$ of 
interior edges.
By applying the construction of completely integrable systems in 
\cite{MR705993},
we obtain the following:

\begin{theorem}[{\cite[Section 4]{MR3211821}}]
For a triangulation $\Gamma$ of the reference polygon, 
the map
\begin{equation}
 \Psi_\Gamma
  =  (\psi_{ij})_{\edg(i,j) \in \Prn \Gamma}
  \colon \scO_{\lambda} \to 
  \bR^{\Prn \Gamma} = \bR^{2n-4}
\label{eq:def_Psi_Gamma}
\end{equation}
is a completely integrable on $\Gr(2,n) \cong \scO_{\lambda}$
with respect to the Kostant-Kirillov form.
The natural coordinate $(u_{ij})_{\edg(i,j) \in \Prn \Gamma}$ 
on $\bR^{\Prn \Gamma}$ gives an action coordinate,
and the image 
$\Delta_{\Gamma} = \Psi_{\Gamma}(\scO_{\lambda})$
is a convex polytope. 
\end{theorem}

\begin{remark}
One can apply this construction for general 
partial flag manifolds of type A to obtain 
several completely integrable systems.
\end{remark}

\begin{remark}
In \cite{MR3211821},
we defined $\psi_{ij} = \lambda^{(i,j)}_2$ for $|i-j| \ge 2$
instead of $\lambda_1^{(i,j)}$.
Since
\begin{equation}
  \lambda_1^{(i,j)} + \lambda_2^{(i,j)}
  = \tr \mu_{G(i,j)}
  = \sum_{k=i}^{j-1} \psi_{k, k+1},
\end{equation}
the completely integrable system $\Psi_{\Gamma}$ 
in \eqref{eq:def_Psi_Gamma} and that in \cite{MR3211821}
are related by a linear transformation on the base space,
and hence fibers of these integrable systems are the same.
\end{remark}

To describe the polytope $\Delta_{\Gamma}$ explicitly,
we recall \emph{bending Hamiltonians}
on \emph{polygon spaces} introduced by
Kapovich and Millson \cite{Kapovich-Millson_SGPES}
and Klyachko \cite{Klyachko_SP}. 
For an $n$-tuple 
$\bsr = (r_1, \dots, r_n) \in (\bR^{>0})^n$
of positive numbers,
the polygon space $\scM_{\bsr}$ is defined to be 
a moduli space of $n$-gons in $\bR^3$ with fixed 
side lengths $r_1, \dots, r_n$:
\begin{equation}
  \scM_{\bsr} \cong 
  \biggl\{ \bsxi = (\xi_1, \dots, \xi_n) \in \prod_{i=1}^n S^2(r_i)
  \biggm| \sum_{i=1}^n \xi_i = 0 \biggm\} \biggr/ SO(3),
\end{equation}
where $S^2(r) \subset \bR^3$ is a 2-sphere of radius $r$
centered at the origin.
For each $1 \le i < j \le n$, 
let $\varphi_{ij} \colon \scM_{\bsr} \to \bR$ be the function
which measures  the length of the diagonal or the side 
of each polygon $\bsxi$ connecting the $i$-th and $j$-th vertices:
\begin{equation}
  \varphi_{ij}(\bsxi) = |\xi_i + \xi_{i+1} + \dots + \xi_{j-1}|.
\end{equation}
Note that $\varphi_{i,i+1} = r_i$ are constant functions.
The function $\varphi_{ij}$ for $|i-j|\ge 2$ is called 
a bending Hamiltonian, since its Hamiltonian flow bends polygons
$\bsxi \in \scM_{\bsr}$ along the diagonal
connecting the $i$-th and $j$-th vertices.

\begin{theorem}
[Kapovich and Millson \cite{Kapovich-Millson_SGPES},
Klyachko \cite{Klyachko_SP}]
For each triangulation $\Gamma$ of $P$,
the map
\begin{equation}
 \Phi_\Gamma
  =  (\varphi_{ij})_{\edg(i,j) \in \Int \Gamma}
  \colon \scM_{\bsr} \to \bR^{n-3}
\end{equation}
is a completely integrable system on $\scM_{\bsr}$.
The image 
$\Phi_{\Gamma}(\scM_{\bsr})$
is a convex polytope defined by triangle inequalities
\begin{equation}
  |\varphi_{ij} - \varphi_{jk}| \le \varphi_{ik}
  \le \varphi_{ij} + \varphi_{jk}
\end{equation}
for each triangle with vertices 
$1 \le i < j < k \le n$ in the triangulation $\Gamma$.
\end{theorem}

The Grassmannian $\Gr(2,n)$ is obtained as 
a symplectic reduction of the space 
$\Mat_{n \times 2}(\bC) \cong (\bC^2)^n$
of $n \times 2$ matrices by the right $U(2)$-action, 
and hence 
the Gelfand-MacPherson correspondence 
\cite{Gelfand-MacPherson_GGGD} 
gives an isomorphism between
the polygon space and a symplectic reduction of 
$\Gr(2,n)$ by the $\bT_{U(n)}$-action.
For a later use, we describe the isomorphism explicitly.
Since the moment map of of the right $U(2)$-action 
on $\Mat_{n \times 2}(\bC)$ is given by
\begin{equation} \label{eq:U(2)-moment}
  \mu_{U(2)} \colon \Mat_{n \times 2}(\bC) \to \sqrt{-1} \fraku (2),
  \quad
  \begin{pmatrix}
    z_1 & w_1 \\
    \vdots & \vdots \\
    z_n & w_n
  \end{pmatrix}
  \mapsto \frac 12
  \sum_{i=1}^n \begin{pmatrix}
                  |z_i|^2 & \overline{z}_i w_i\\
                  z_i \overline{w}_i & |w_i|^2
               \end{pmatrix},
\end{equation}
the level set 
 $\mu_{U(2)}^{-1}(\lambda \bsone_2)$ 
of $\mu_{U(2)}$ consists of 
$(z_i, w_i)_i \in \Mat_{n \times 2}(\bC)$ satisfying
\begin{equation}
  \sum_{i=1}^n |z_i|^2 = \sum_{i=1}^n |w_i|^2 =  2 \lambda,
  \quad
  \sum_{i=1}^n z_i \overline{w}_i = 0,
  \label{eq:ONB}
\end{equation}
and hence
\begin{equation}
 Z = (z_i, w_i)_i \longmapsto 
  \frac 12 Z Z^* = \frac 12 (z_i \overline{z_j} + w_i \overline{w_j})_{i,j}
\end{equation}
gives an isomorphism 
$\mu_{U(2)}^{-1}(\lambda \bsone_2) / U(2)
\to \scO_{\lambda} \cong \Gr(2,n)$.
The moment map
$\mu_{\bT_{U(n)}} \colon  \Mat_{n \times 2}(\bC) \to \bR^n$ 
of the left $\bT_{U(n)}$-action 
on $\Mat_{n \times 2}(\bC)$ is given by
\begin{equation} \label{eq:T-moment}
  \begin{pmatrix}
    z_1 & w_1 \\
    \vdots & \vdots \\
    z_n & w_n
  \end{pmatrix}
  \longmapsto
  \left( \frac{|z_1|^2 + |w_1|^2}2, \dots ,
  \frac{|z_n|^2 + |w_n|^2}2 \right),
\end{equation}
and thus the projection
\begin{equation}
  \mu_{\bT_{U(n)}}^{-1}(2 \bsr) 
  \cong \prod_{i=1}^n S^3(2 \sqrt{r_i})
  \longrightarrow 
  \bT_{U(n)} \backslash \mu_{\bT_{U(n)}}^{-1}(2 \bsr) 
  \cong \prod_{i=1}^n S^2(r_i)
\end{equation}
is written as
\begin{equation}
    \begin{pmatrix}
    z_1 & w_1 \\
    \vdots & \vdots \\
    z_n & w_n
  \end{pmatrix}
  \longmapsto
  (\nu(z_1, w_1), \dots, \nu(z_n, w_n) )
  \label{eq:GIT_S^2}
\end{equation}
by using the Hopf fibration 
\begin{equation}
  \nu \colon S^3(2 \sqrt{r}) \to S^2(r), \quad
  (z, w) \mapsto 
  \left( \frac{z \overline{w}}{2}, \frac{|z|^2-|w|^2}{4} \right),
  \label{eq:Hopf}
\end{equation}
where we regard $S^3(2 \sqrt{r}) \subset \bC^2$ and
$S^2(r) \subset \bC \times \bR$.
Since the condition \eqref{eq:ONB} implies 
$\sum_i \nu(z_i, w_i) = 0$, 
the map \eqref{eq:GIT_S^2} induces an isomorphism 
\begin{align}
  \bT_{U(n)} \GITl_{2\bsr} \Gr(2,n)
  & \cong  
  \bT_{U(n)} \backslash 
   \bigl( \mu_{U(2)}^{-1}(\lambda \bsone_2)
  \cap \mu_{\bT_{U(n)}}^{-1} (2\bsr) \bigr)  / U(2) \\
  &\cong  \left( \prod_{i=1}^n S^2(r_i) \right) 
  \biggr/ \!\!\!\!\! \biggr/_{\hspace{-.5em} 0} SU(2) 
  = \scM_{\bsr}.
\end{align}

Let $\varphi_{ij}$ also denote the pull-back 
to $\Gr(2,n)$ of the bending Hamiltonian.

\begin{proposition}[{\cite[Proposition 4.6]{MR3211821}}]
Two completely integrable systems 
$(\psi_{ij})_{\edg(i,j) \in \Prn \Gamma}$ and 
$(\varphi_{ij})_{\edg(i,j) \in \Prn \Gamma}$ are related by
\begin{equation}
  \varphi_{ij} = 
    \psi_{ij} - \frac 12 \sum_{k=i}^{j-1} \psi_{k,k+1},
\label{eq:GC-bending}
\end{equation}
and thus $\Psi_{\Gamma}$
induces $\Phi_{\Gamma}$ on each polygon space $\scM_{\bsr}$
under the symplectic reduction.
\end{proposition}

Note that $\varphi_{i,i+1} = (1/2)\psi_{i,i+1}$ is not an action coordinate,
since its Hamiltonian flow has period $\pi (\ne 2 \pi)$.
We give a proof of this proposition for readers' convenience.

\begin{proof}
We first note that the function $\psi_{i, i+1}$ associated to 
the leaf $\epsilon(i, i+1)$ is given by
\begin{equation}
  \psi_{i, i+1}([z_k, w_k])  = \frac{|z_i|^2 + |w_i|^2}{2},
  \quad [z_k, w_k]_k \in \Gr(2,n),
\end{equation}
which coincides with $2 \varphi_{i, i+1}([z_k, w_k])$.
Recall that, for $|i-j| \ge 2$, 
\begin{equation}
\psi_{ij}([z_k, w_k]) 
= \lambda_1^{(i,j)}([z_k, w_k]) \ge \lambda_2^{(i,j)}([z_k, w_k]) \ge 0
\end{equation}
are the first and second eigenvalues of 
\begin{equation}
\frac 12
\begin{pmatrix}
  z_i & w_i \\
  \vdots & \vdots \\
  z_{j-1} & w_{j-1}
\end{pmatrix}
\begin{pmatrix}
  \overline{z}_i & \hdots & \overline{z}_{j-1} \\
  \overline{w}_i & \hdots & \overline{w}_{j-1}
\end{pmatrix}.
\end{equation}
Then the $2 \times 2$ Hermitian matrix
\begin{multline}
\frac 12 
\begin{pmatrix}
  \overline{z}_i & \hdots & \overline{z}_{j-1} \\
  \overline{w}_i & \hdots & \overline{w}_{j-1}
\end{pmatrix}
\begin{pmatrix}
  z_i & w_i \\
  \vdots & \vdots \\
  z_{j-1} & w_{j-1}
\end{pmatrix}
= \frac 12 \sum_{k=i}^{j-1}
\begin{pmatrix}
  |z_k|^2 & \overline{z}_k w_k \\
  z_k \overline{w}_k & |w_k|^2
\end{pmatrix} \\
=
\frac 14 \sum_{k=i}^{j-1}
\begin{pmatrix}
  |z_k|^2 - |w_k|^2 & 2 \overline{z}_k w_k \\
  2 z_k \overline{w}_k & |w_k|^2 - |z_k|^2
\end{pmatrix}
+
\sum_{k=i}^{j-1} \frac{|z_k|^2 + |w_k|^2}{4} 
\begin{pmatrix} 1 & 0 \\ 0 & 1 \end{pmatrix}
\end{multline}
has eigenvalues $ \lambda_1^{(i,j)} \ge \lambda_2^{(i,j)} $.
Since 
\begin{equation}
\frac 14 \sum_{k=i}^{j-1}
\begin{pmatrix}
  |z_k|^2 - |w_k|^2 & 2 \overline{z}_k w_k \\
  2 z_k \overline{w}_k & |w_k|^2 - |z_k|^2
\end{pmatrix}
\in \sqrt{-1} \mathfrak{su} (2)
\end{equation}
has eigenvalues 
$\pm \| \sum_{k=i}^{j-1} \nu (z_k, w_k) \| = \pm \varphi_{ij} ([z_k, w_k])$, 
we have
\begin{align}
  \lambda_1^{(i,j)} &= \varphi_{ij} + 
  \sum_{k=i}^{j-1} \frac{|z_k|^2 + |w_k|^2}{4} ,\\
  \lambda_2^{(i,j)} &= - \varphi_{ij} + 
  \sum_{k=i}^{j-1} \frac{|z_k|^2 + |w_k|^2}{4},
\end{align}
which prove the proposition.
\end{proof}

We introduce another coordinate 
$(u(i,j))_{\edg(i,j) \in \Prn \Gamma}$ 
on $\bR^{2n-4}$ corresponding to 
$(\varphi_{ij})_{\edg (i,j) \in \Prn \Gamma}$ defined by
\begin{equation}
  u(i,j) = 
    u_{ij} - \frac 12 \sum_{k=i}^{j-1} u_{k,k+1}.
  \label{eq:bending_coord}
\end{equation}

\begin{corollary}
The moment polytope
$\Delta_{\Gamma} = \Psi_{\Gamma}(\scO_{\lambda})$
is  defined by triangle inequalities
\begin{equation}
  |u(i,j) - u(j,k)| \le u(i,k) \le u(i,j) + u(j,k)
  \label{eq:triangle_ineq}
\end{equation}
for each triangle with vertices 
$1 \le i < j < k \le n$ in the triangulation $\Gamma$.
\end{corollary}
In terms of the action coordinate 
$(u_{ij})_{\edg(i,j) \in \Prn \Gamma}$,
the inequalities \eqref{eq:triangle_ineq}   are written as
\begin{align}
  u_{ik} & \ge u_{ij} - u_{jk} + \sum_{l=j}^{k-1} u_{l, l+1}, 
    \label{eq:triangle_ineq1} \\
  u_{ik} & \ge u_{jk} - u_{ij} + \sum_{l=i}^{j-1} u_{l, l+1}, 
    \label{eq:triangle_ineq2} \\
  u_{ik} & \le u_{ij} + u_{jk}.
    \label{eq:triangle_ineq3}
\end{align}

\begin{example}[Hausmann and Knutson 
\cite{Hausmann-Knutson_PSG}]
\begin{figure}[h]
\centering
\includegraphics[bb=0 0 96 100]{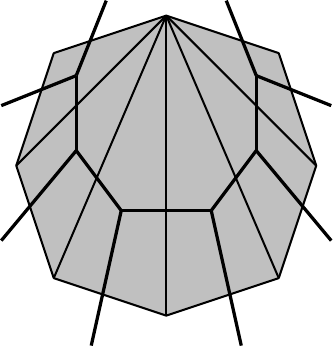}
\caption{The caterpillar}
\label{fg:caterpillar}
\end{figure}
For the triangulation $\Gamma_\cat$ given by
\begin{align} \label{eq:caterpillar}
 \Int \Gamma_\cat
  := \lc \edg(1,3), \edg(1,4), \ldots, \edg(1,n-2) \rc
\end{align}
shown in \pref{fg:caterpillar}
called the \emph{caterpillar},
the integrable system $\Psi_\GC := \Psi_{\Gamma_\cat}$
gives the {\em Gelfand-Cetlin system}
introduced by Guillemin and Sternberg
\cite{MR705993}, 
%
and the triangle inequalities \eqref{eq:triangle_ineq} 
give the Gelfand-Cetlin pattern
\begin{align} \label{eq:GC}
\begin{alignedat}{17}
  \lambda &&&& \gcbox{\sum_{i=1}^{n-1} \psi_{i,i+1} - \lambda} \\
  & \uge && \dge && \uge &&&&&&&&&&& \\
  && \gcbox{\psi_{1,n-1}} &&&& 
     \gcbox{\sum_{i=1}^{n-2}\psi_{i,i+1}-\psi_{1,n-1}} \\
  &&& \uge && \dge && \uge &&&&&&&&& \\
  &&&& \gcbox{\psi_{1,n-2}} &&&& 
    \gcbox{\sum_{i=1}^{n-3} \psi_{i,i+1} - \psi_{1,n-2}} \\
  &&&&& \uge && \dge && \uge &&&&&&& \\
  &&&&&& \dndots &&&& \dndots &&&& 0 && \\
  &&&&&&& \uge && \dge && \uge && \dge &&& \\
  &&&&&&&& \gcbox{\psi_{13}} &&&& 
    \gcbox{\sum_{i=1}^{2}\psi_{i,i+1}-\psi_{13}} &&&&& \\
  &&&&&&&&& \uge && \dge &&&&&& \\
  &&&&&&&&&& \gcbox{\psi_{12}} &&&&&&& \\
\end{alignedat}.
\end{align}

\end{example}

Suppose that we have two triangulations $\Gamma$, $\Gamma'$ 
of $P$ which are related by a \emph{Whitehead move}
in a quadrilateral $P_0$ 
with vertices  $1 \le a<b<c<d \le n$
(see Figure \ref{fg:flip1}, where $P_0$ is unshaded).
\begin{figure}[h]
\centering
\includegraphics[bb=0 0 216 88]{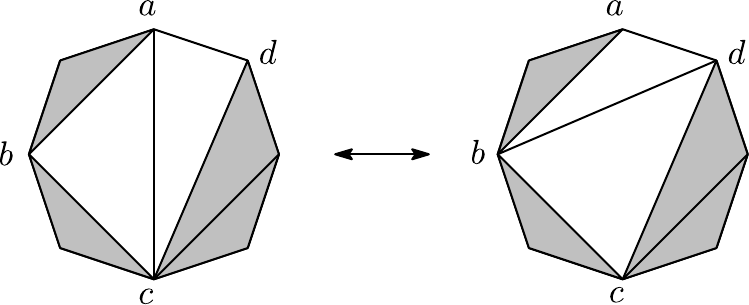}
\caption{A whitehead move}
\label{fg:flip1}
\end{figure}
Let $\Gamma''$ be the subdivision of $P$ given by common diagonals
in $\Gamma$ and $\Gamma'$  (see Figure \ref{fg:flip2});
its dual graph is obtained 
from that of $\Gamma$ (resp. $\Gamma'$)
by contracting the edge $\edg(a,c) \in \Gamma$
(resp. $\edg(b,d) \in \Gamma'$).
\begin{figure}[h]
\centering
\includegraphics[bb=0 0 102 103]{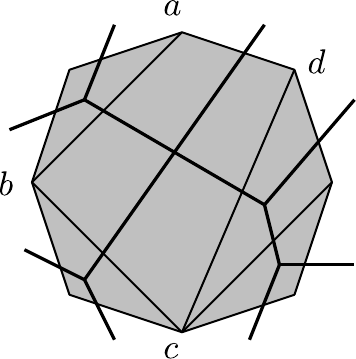}
\caption{A subdivision of $P$ given by common diagonals
           in two triangulations in Figure \ref{fg:flip1}.}
\label{fg:flip2}
\end{figure}
Note that the action coordinates on 
$\Delta_{\Gamma} \subset \bR^{\Prn \Gamma}$ and 
$\Delta_{\Gamma'} \subset \bR^{\Prn \Gamma'}$ are written as
\begin{equation}
  \bsu = ((u_{ij})_{\edg(i,j) \in \Prn \Gamma''}, u_{ac}),
  \quad
  \bsu' = ((u_{ij})_{\edg(i,j) \in \Prn \Gamma''}, u_{bd}),
\end{equation}
respectively.

\begin{proposition}
Let $\Gamma, \Gamma'$ be two triangulations of $P$ as above.
Then the piecewise linear transform
$\bsu \mapsto \bsu'$ defined by
\begin{equation}
  u_{bd} = - u_{ac} + u_{ab} + u_{bc} + u_{cd} + u_{ad}
  - \min \left( u_{ab}+u_{cd}, u_{ad} + u_{bc} 
  - \sum_{i=b}^{c-1} u_{i, i+1}
    \right)
\label{eq:PL-transform}
\end{equation}
gives a bijection between 
$\Delta_{\Gamma}$ and $\Delta_{\Gamma'}$.
\end{proposition}

\begin{remark}
The piecewise linear transformation \eqref{eq:PL-transform}
is different from the one
given in \cite[Proposition 3.5]{MR3211821}.
\end{remark}

\begin{proof}
For each fixed 
$(\varphi_{ab}, \varphi_{bc}, \varphi_{cd}, \varphi_{ad})
= (r_1, r_2, r_3, r_4) \in (\bR^{>0})^4$, 
the ranges of the bending Hamiltonians 
$\varphi_{ac}$ and $\varphi_{bd}$ are given by
\begin{align}
  \max \{ |r_1-r_2|, |r_3-r_4| \}
  \le u(a,c) \le 
  \min \{r_1+r_2, r_3+r_4 \}, \\
  \max \{ |r_1-r_4|, |r_2-r_3| \}
  \le u(b,d) \le 
  \min \{r_1+r_4, r_2+r_3 \},
\end{align}
respectively.
Since 
\begin{align}
  &- \max \{ |r_1 - r_2|, |r_3 - r_4| \} \\
  & \quad = \min \{ \min\{ r_1 - r_2, r_3 - r_4 \}, 
                  \min \{r_2 - r_1, r_4 - r_3 \} \} \\
  & \quad = \min \{ \min \{r_1 + r_4, r_2 + r_3 \} - r_2- r_4, 
                 \min \{r_1 + r_4, r_2 + r_3 \} - r_1- r_3 \} \\
  & \quad =  \min \{r_1 + r_4, r_2 + r_3 \}
      + \min \{  - r_2- r_4,- r_1- r_3 \} \\
  & \quad = \min \{r_1 + r_4, r_2 + r_3 \} + \min \{ r_1 + r_3, r_2 + r_4 \}
      - ( r_1 + r_2 + r_3 + r_4 ),
\end{align}
the lengths of the ranges are the same;
\begin{multline}
  \min \{r_1+r_2, r_3+r_4 \} 
  - \max \{ |r_1-r_2|, |r_3-r_4| \} \\
  = 
  \min \{r_1+r_4, r_2+r_3 \}
  - \max \{ |r_1-r_4|, |r_2-r_3| \}.
\end{multline}
Thus the map $u(a,c) \mapsto u(b,d)$ defined by
\begin{align}
  u(b,d) &= - u(a,c) +  \min \{r_1 +r_2 , r_3 +r_4 \} 
          + \max \{ |r_1-r_4|, |r_2-r_3| \} 
  \notag \\
  &= - u(a,c) + r_1 + r_2 + r_3 + r_4 
       - \min \{ r_1+r_3, r_2 + r_4 \}
  \label{eq:PL-transform2}
\end{align} 
gives a bijection between the ranges of 
$\varphi_{ac}$ and $\varphi_{bd}$.
One can easily check that
this map is written as \eqref{eq:PL-transform}
under the coordinate change \eqref{eq:bending_coord}.
\end{proof}

\section{Degenerations of Grassmannians}

Speyer and Sturmfels \cite{Speyer-Sturmfels_TG} have shown that 
toric degenerations of the Grassmannian $\Gr(2,n)$ are
parametrized by the tropical Grassmannian, 
and its top dimensional cells are in one-to-one correspondence 
with the set of trivalent trees $\Gamma$ with $n$-leaves.
For each $\Gamma$, the corresponding toric degeneration
$f_{\Gamma} \colon \frakX_{\Gamma} \to \bC^{n-3}$ 
of $\Gr(2,n)$ can be constructed as follows.
Let $\bsp = [p_{\lf{i}\lf{j}}]_{1 \le \lf{i} < \lf{j} \le n}$ be 
a homogeneous coordinate on
$\bP(\bigwedge^2 \bC^n)$ so that 
$\Gr(2,n) \subset \bP(\bigwedge^2 \bC^n)$ is given by
the Pl\"ucker relations
\begin{equation}
  F_{\lf{i}\lf{j}\lf{k}\lf{l}}(\bsp) = 
  p_{\lf{i} \lf{j}}p_{\lf{k} \lf{l}} - p_{\lf{i}\lf{k}}p_{\lf{j}\lf{l}} 
  + p_{\lf{i}\lf{l}}p_{\lf{j}\lf{k}} = 0
\end{equation}
for $1 \le i < j < k < l \le n$.\footnote{%
Note that the indices of the the Pl\"ucker coordinates 
are labels of leaves of $\Gamma$ (or equivalently, 
sides of the reference polygon $P$), while
the indices of the Hamiltonians $\psi_{ij}$
(and hence, those of coordinates on the SYZ mirror) are 
labels of vertices of $P$.
}
For each $1 \le i , j \le n$, let $\gamma(\lf{i}, \lf{j})$ denote 
the path in the tree $\Gamma$ connecting the $i$-th and 
$j$-th leaves $\edg(i, i+1)$, $\edg(j, j+1)$.
\begin{figure}[h]
  \centering
  \includegraphics[bb=0 0 147 105]{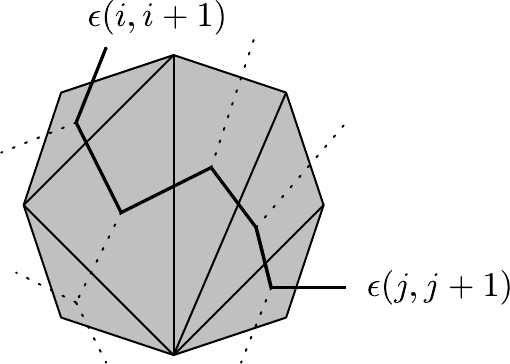}
  \caption{A path $\gamma(i,j)$ connecting the $i$-th and $j$-th leaves}
\end{figure}
We introduce $n-3$ deformation parameters 
$\bst = (t_{kl})_{\edg(k,l) \in \Int \Gamma} 
\in \bC^{\Int \Gamma}$.
Define a weight 
$\bsw (\lf{i}, \lf{j}) = (w_{kl}(\lf{i}, \lf{j}))_{\edg(k,l) \in \Int \Gamma} 
\in \bQ^{\Int \Gamma}$
of a Pl\"ucker coordinate $p_{\lf{i} \lf{j}}$ by
\begin{equation}
  w_{kl}(\lf{i}, \lf{j}) = \begin{cases}
    1/2 \quad 
      & \text{if $\gamma (\lf{i}, \lf{j})$ contains the edge $\edg(k,l)$,}\\
    0 & \text{otherwise},
  \end{cases}
\end{equation}
and consider an action of $(\bCx)^{\Int \Gamma}$ on 
$\bigwedge^2 \bC^n$ given by
\begin{equation}
  \bst \cdot \bsp = 
  ( \bst^{\bsw(\lf{i}, \lf{j})} p_{\lf{i} \lf{j}})
  = \left( \prod_{\edg(k,l)} t_{kl}^{w_{kl}(\lf{i}, \lf{j})}
     p_{\lf{i} \lf{j}} \right)
  = \left( \prod_{\edg(k,l) \subset \gamma(\lf{i}, \lf{j})}
     t_{kl}^{1/2} p_{\lf{i} \lf{j}} \right). 
\end{equation}
For a polynomial
\begin{equation}
  F(\bsp) = \sum_{I= (i_1, j_1, \dots, i_m, j_m)} 
    c_I p_{\lf{i_1} \lf{j_1}} \dots p_{\lf{i_m} \lf{j_m}}
\end{equation} 
in Pl\"ucker coordinates,
let
\begin{equation}
  w_{kl}(F) = \max \{ 
  w_{kl}(\lf{i_1}, \lf{j_1}) + \dots + w_{kl}(\lf{i_m}, \lf{j_m}) 
  \mid  c_I  \ne 0 \}
\end{equation}
be the maximum of weights of monomials in $F(\bsp)$
with respect to $t_{kl}$, and define
$F^{\Gamma}(\bsp, \bst) \in \bC[ \bsp, \bst]$ by
\begin{equation}
  F^{\Gamma} (\bsp, \bst) = 
  \bst^{\bsw(F)}
  F ( \bst^{-\bsw(\lf{i}, \lf{j})}p_{\lf{i} \lf{j}}),
\end{equation}
where 
\begin{equation}
  \bst^{\bsw(F)}  = \prod_{\edg(k,l) \in \Int \Gamma} 
  t_{kl}^{w_{kl}(F)}.
\end{equation}
Then the degenerating family 
$f_{\Gamma} \colon \frakX_{\Gamma} \to \bC^{n-3}$
associated with $\Gamma$ is given by
\begin{equation}
\frakX_{\Gamma} = 
\{ (\bsp, \bst) \in \bP( \textstyle{\bigwedge^2} \bC^n) 
\times \bC^{n-3}
\mid F^{\Gamma}_{\lf{i} \lf{j} \lf{k} \lf{l}}(\bsp, \bst) = 0, \ i<j<k<l \},
\label{eq:family_X_Gamma}
\end{equation}
whose central fiber 
$X_{\Gamma} = f_{\Gamma}^{-1}(0, \dots, 0)$ is a toric variety 
with moment polytope $\Delta_{\Gamma}$
(see \cite[Example 5.2]{MR3211821} for the case $n=5$).
\cite[Theorem 1.2]{MR3211821}
combined with \cite[Theorem 5.4]{MR3425384}
gives the following.

\begin{theorem}
\label{th:toric_deg}
The completely integrable system $\Psi_{\Gamma}$ 
on $\Gr(2,n)$ can be deformed into 
a toric moment map $\mu_{\bT_{\Gamma}}$ on $X_{\Gamma}$
with moment polytope $\Delta_{\Gamma}$.
There exists a map $\phi \colon \Gr(2,n) \to X_{\Gamma}$
which sends each Lagrangian torus fiber 
$L_{\Gamma}(\bsu) = \Psi_{\Gamma}^{-1}(\bsu)$ 
diffeomorphically to the fiber $\mu_{\bT_{\Gamma}}^{-1}(\bsu)$
of $\mu_{\bT_{\Gamma}}$ 
over the same point $\bsu \in \Int \Delta_{\Gamma}$.
\end{theorem}

The map $\phi \colon \Gr(2,n) \to X_{\Gamma}$ is
given by the \textit{gradient-Hamiltonian flow} for $f_{\Gamma}$
introduced by Ruan \cite{Ruan_I}.
Harada and Kaveh  \cite[Theorem 5.4]{MR3425384}
show that the gradient-Hamiltonian flow extends to 
singular loci of $X_{\Gamma}$.

Let $(X, \omega)$ be a symplectic manifold,
and fix a compatible almost complex structure $J$.
For a Lagrangian submanifold $L$ in $X$ 
and a relative homotopy class $\beta \in \pi_2(X, L)$,
let $\overline{\scM}_1(X,L; \beta)$ 
denote the  the moduli space of stable $J$-holomorphic maps of degree $\beta$ 
from a bordered Riemann surface of genus zero 
with one marked point and
with Lagrangian boundary condition.
Fix
$\bst \in (\bC^{\times})^{n-3} \subset \bC^{n-3}$
sufficiently close to the origin, 
and let $J_{\bst}$ be the complex structure on the fiber
$X_{\bst} = f^{-1}_{\Gamma}(\bst)$ 
of the family \eqref{eq:family_X_Gamma}.
By using a symplectomorphism $\Gr(2,n) \to X_{\bst}$
given by the gradient-Hamiltonian flow,
we regard $J_{\bst}$ as a complex structure on $\Gr(2,n)$.
The fact that the toric variety $X_{\Gamma}$ is  Fano and 
admits a small resolution \cite[Section 8]{MR3211821}
enable us to apply the argument in  
\cite[Section 9]{MR2609019}
to obtain the following:

\begin{theorem}[{\cite[Proposition 9.16]{MR2609019}}] 
\label{th:toric_deg1}
For  each $\bsu \in \Int \Delta_{\Gamma}$ 
and 
$\beta \in \pi_2(\Gr(2,n), L_{\Gamma}(\bsu))$ of Maslov index two,
there exists a diffeomorphism 
\begin{equation}
  \overline{\scM}_1(\Gr(2,n), L_{\Gamma}(\bsu); \beta)
  \overset{\sim}{\to}  
  \overline{\scM}_1(X_{\Gamma}, 
    \mu_{\bT_{\Gamma}}^{-1}(\bsu); \beta)
  \label{eq:diff_moduli}
\end{equation}
such that
\begin{equation}
  \begin{CD}
    H_*(\overline{\scM}_1(\Gr(2,n), L_{\Gamma}(\bsu); \beta)) 
    @>{\ev_*}>> 
   H_*(L_{\Gamma}(\bsu)) \\
   @VVV @VV{\phi_*^{-1}}V \\
    H_*( \overline{\scM}_1(X_{\Gamma}, 
    \mu_{\bT_{\Gamma}}^{-1}(\bsu); \beta) )
    @>{\ev_*}>> H_*(\mu_{\bT_{\Gamma}}^{-1}(\bsu))
  \end{CD}
\end{equation}
commutes,
where $\ev \colon \overline{\scM}_1(X, L; \beta) \to L$ is the
evaluation map at boundary marked points,
and $\beta$ in the right hand side of \eqref{eq:diff_moduli} 
is regarded as a homotopy class
in $X_{\Gamma}$ via the isomorphism 
$\phi_* \colon \pi_2(\Gr(2,n), L_{\Gamma}(\bsu)) \to 
  \pi_2(X_{\Gamma}, \mu_{\bT_{\Gamma}}^{-1}(\bsu))$.
\end{theorem}

Suppose that we have two triangulations 
$\Gamma$ and $\Gamma'$ related by 
a Whitehead move as in Figure \ref{fg:flip1}.
Then the subdivision $\Gamma''$ defined by
common diagonals in $\Gamma$ and $\Gamma'$ 
gives  a subfamily 
\begin{equation}
  f_{\Gamma''} \colon \frakX_{\Gamma''} \longrightarrow \bC^{n-4} 
  = \bC^{\Int \Gamma''}
  \label{eq:partial_deg}
\end{equation}
of 
$f_{\Gamma} \colon \frakX_{\Gamma} \to \bC^{\Int \Gamma}$ 
(resp. 
$f_{\Gamma'} \colon \frakX_{\Gamma'} \to \bC^{\Int \Gamma'}$)
defined by $t_{ac} = 1$ (resp. $t_{bd}=1$),
on which 
$\Psi_{\Gamma} = (\psi_{ij})_{\edg(i,j) \in \Prn \Gamma}$ 
(resp. 
$\Psi_{\Gamma'} = (\psi_{ij})_{\edg(i,j) \in \Prn \Gamma'}$) 
can be deformed into 
a completely integrable system 
$\Psi_{\Gamma}^0 
  = (\psi_{ij}^0)_{\edg(i,j) \in \Prn \Gamma}$ 
(resp. 
$\Psi_{\Gamma'} = (\psi_{ij}^0)_{\edg(i,j) \in \Prn \Gamma'}$) 
on the central fiber $X_0 = f_{\Gamma''}^{-1}(0, \dots, 0)$.
We write fibers of $\Psi_{\Gamma}^0$ as
$L_{\Gamma}^0(\bsu) = (\Psi_{\Gamma}^0)^{-1}(\bsu)$.

\begin{corollary} \label{cr:toric_deg2}
For each $\bsu \in \Int \Delta_{\Gamma}$ and 
$\beta \in \pi_2(\Gr(2,n), L_{\Gamma}(\bsu))$ of Maslov index two,
there exists a diffeomorphism 
\begin{equation}
  \overline{\scM}_1(\Gr(2,n), L_{\Gamma}(\bsu); \beta)
  \overset{\sim}{\to}  
  \overline{\scM}_1(X_0, L^0_{\Gamma}(\bsu); \beta)
\end{equation}
which commutes with evaluation maps on the homology groups.
The same is true for $\Gamma'$.
\end{corollary}

We observe defining equations for $X_0$ are:
\begin{equation}
  F^{\Gamma''}_{\lf{i} \lf{j} \lf{k} \lf{l}} (\bsp, \bszero)
= F^{\Gamma}_{\lf{i} \lf{j} \lf{k} \lf{l}} 
  (\bsp, (0, \dots, 0, t_{ac}=1)) = 0.
\end{equation}
If the paths $\gamma(\lf{i}, \lf{k})$ and $\gamma(\lf{j}, \lf{l})$
intersect transversally in the interior of the quadrilateral $P_0$
as in Figure \ref{fg:X-path},
then all monomials in $F_{\lf{i} \lf{j} \lf{k} \lf{l}}$ have the same weight,
and thus
the Pl\"ucker relation is unchanged:
\begin{equation}
  F^{\Gamma''}_{\lf{i} \lf{j} \lf{k} \lf{l}} 
    (\bsp, \bszero)
  = F_{\lf{i} \lf{j} \lf{k} \lf{l}} (\bsp)
  = p_{\lf{i} \lf{j}} p_{\lf{k} \lf{l}} - p_{\lf{i} \lf{k}} p_{\lf{j} \lf{l}}
     + p_{\lf{i} \lf{l}} p_{\lf{j} \lf{k}}.
\end{equation}
In the case where $\gamma(\lf{i}, \lf{k})$ and $\gamma(\lf{j}, \lf{l})$
share at least one interior edge, the Pl\"ucker relation is 
deformed into a binomial
\begin{equation}
  F^{\Gamma''}_{\lf{i} \lf{j} \lf{k} \lf{l}} 
    (\bsp, \bszero)
  =  - p_{\lf{i} \lf{k}} p_{\lf{j} \lf{l}} + p_{\lf{i} \lf{l}} p_{\lf{j} \lf{k}},
\label{eq:pluecker-binomial}
\end{equation}
where we assume that 
$\gamma(\lf{i}, \lf{j})$ and $\gamma(\lf{k}, \lf{l})$
do not share any edge in $\Gamma''$
(see Figure \ref{fg:H-path}).
\begin{figure}[h]
  \centering
  \begin{minipage}{5cm}
    \centering
    \includegraphics[bb=0 0 108 97]{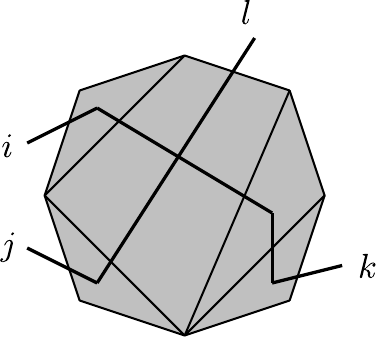}
    \caption{$\gamma(\lf{i}, \lf{k})$ and $\gamma(\lf{j}, \lf{l})$
                intersect transversally.}
    \label{fg:X-path}
  \end{minipage}
  \hspace{1cm}
  \begin{minipage}{5cm}
    \centering
    \includegraphics[bb=0 0 103 96]{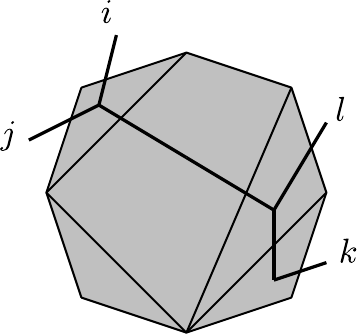}
    \caption{$\gamma(\lf{i}, \lf{k})$ and $\gamma(\lf{j}, \lf{l})$
                share an interior edge in $\Gamma''$.}
    \label{fg:H-path}
  \end{minipage}
\end{figure}

We give a description of $X_0$ following an idea 
in \cite{Howard-Manon-Millson}.
By cutting the reference polygon $P$ 
along the diagonals in $\Gamma''$,
we obtain a subdivision of $P$
into one quadrilateral $P_0$ and 
$n-4$ triangles $P_1, \dots , P_{n-4}$,
and the corresponding forest (i.e., a set of trees)
$\Gamma_0, \Gamma_1, \dots, \Gamma_{n-4}$.
\begin{figure}[h]
  \centering
    \includegraphics[bb=0 0 121 116]{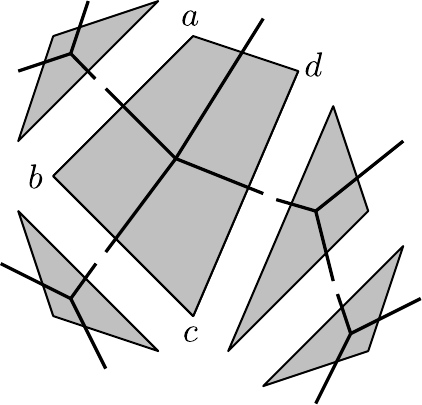}
  \caption{A subdivision of $P$ and
              the corresponding forest.}
  \label{fg:forest}
\end{figure}
For each subpolygon $P_{\alpha}$, we associate 
a cone 
$\widetilde{\Gr}(P_{\alpha})
\subset \bigwedge^2 \bC^{\Gamma_{\alpha}}$ over
the Grassmannian
\begin{equation}
  \Gr(P_{\alpha})
  := \Gr(2, \bC^{\Gamma_{\alpha}})
  \cong \begin{cases}
    \Gr(2,4) \quad &\alpha = 0,\\
    \bP^2 & \alpha = 1, \dots, n-4,
  \end{cases}
\end{equation}
on which the Pl\"ucker coordinates 
$(p^{P_{\alpha}}_{\edg, \edg'})_{\edg, \edg' \in \Gamma_{\alpha}}$
are indexed by pairs of (boundary) edges of $\Gamma_{\alpha}$.
We consider an action of a torus
$(\bCx)^{\amalg_{\alpha} \Gamma_{\alpha}} 
\cong  (\bCx)^{3n-8}$
on $\prod_{\alpha=0}^{n-4} \widetilde{\Gr}(P_{\alpha})$ defined  by
\begin{equation}
  \bstau \cdot (p_{\edg, \edg'}^{P_{\alpha}})_{\edg, \edg'}
  = ( \tau_{\edg} \tau_{\edg'} p_{\edg, \edg'}^{P_{\alpha}})_{\edg, \edg'}
  \label{eq:T^forest}
\end{equation}
for $\bstau = (\tau_{\edg})_{\edg \in \amalg_{\alpha} \Gamma_{\alpha}}
\in (\bCx)^{\amalg_{\alpha} \Gamma_{\alpha}}$.
We regard 
$(\bCx)^{\partial \Gamma''} = \prod_{i=1}^n\bCx_{\edg_{\lf{i}}}$  
as a subgroup of $(\bCx)^{\amalg_{\alpha} \Gamma_{\alpha}}$
by identifying leaves $\edg_{\lf{i}} = \edg(i,i+1)$ of $\Gamma''$
with corresponding edges in the forest $\amalg_{\alpha} \Gamma_{\alpha}$,
and define $\bCx_{\partial P} \cong \bCx$
to be the diagonal subgroup of 
$(\bCx)^{\partial \Gamma''}  ( \subset 
(\bCx)^{\amalg_{\alpha} \Gamma_{\alpha}})$.
For an interior edge $\edg = \edg(i,j) \in \Int \Gamma''$, 
let $\edg^+, \edg^-$ be two copies of $\edg$ in 
$\amalg_{\alpha} \Gamma_{\alpha}$, and
define 
\begin{equation}
  \bCx_{\edg^+, \edg^-} = \{ (\tau, \tau^{-1}) \in 
  \bCx_{\edg^+} \times \bCx_{\edg^-}
  \mid \tau \in \bCx \}
  \cong \bCx
\end{equation}
to be the anti-diagonal subgroup of 
$\bCx_{\edg^+} \times \bCx_{\edg^-}
(\subset (\bCx)^{\amalg_{\alpha} \Gamma_{\alpha}})$.
Then the torus action \eqref{eq:T^forest}
induces an action of 
the ($n-3$)-dimensional subtorus
of $(\bCx)^{\amalg_{\alpha} \Gamma_{\alpha}}$
\begin{equation}
  \bCx_{\partial P} \times 
  \prod_{\edg \in \Int \Gamma''} \bCx_{\edg^+, \edg^-}
  \cong \bCx \times (\bCx)^{n-4}
\end{equation}
on $\prod_{\alpha} \widetilde{\Gr}(P_{\alpha})$.
We define
\begin{equation}
 \bT_{\Gamma''}^{\bC} = 
  (\bCx)^{\amalg_{\alpha} \Gamma_{\alpha}}  \left/
  \prod_{\edg \in \Int \Gamma''} \bCx_{\edg^+, \edg^-} \right.
  \cong (\bCx)^{\Gamma''} 
  \cong (\bCx)^{2n-4}.
\end{equation}

\begin{proposition}
The central fiber $X_0$ of the family \eqref{eq:partial_deg}
is given by the GIT quotient
\begin{equation}
  X_0 \cong \biggl.\biggl.
    \prod_{\alpha=0}^{n-4} \widetilde{\Gr}(P_{\alpha}) 
  \biggr/ \!\!\!\!\! \biggr/
  \Bigl( \bCx_{\partial P} \times 
  \prod_{\edg \in \Int \Gamma''} \bCx_{\edg^+, \edg^-} \Bigr),
\end{equation}
and the inclusion $X_0 \hookrightarrow \bP(\bigwedge^2 \bC^n)$ 
is given by
\begin{equation}
  p_{\lf{i} \lf{j}} = \prod_{\edg, \edg' \subset \gamma(\lf{i}, \lf{j})} p_{\edg, \edg'}^{P_{\alpha}},
\end{equation}
where the product in the right hand side is taken over 
a sequence of edges of $\amalg_{\alpha} \Gamma_{\alpha}$ 
contained in the path $\gamma(\lf{i}, \lf{j} )$.
Furthermore, the induced action of 
$\bT^{\bC}_{\Gamma''}/ \bCx_{\partial P} \cong (\bCx)^{2n-5}$
on $X_0$ is the complexification of the Hamiltonian torus 
action of $(\psi_{ij}^0)_{\edg (i,j) \in \Prn \Gamma''}$.
\end{proposition}
See \cite[Section 5, 6]{MR3211821} for more detail.
Note that the $(\bCx)^{\partial \Gamma''}$-action on $X_0$
coincides with the complexification of the 
$\bT_{U(n)}$-action.
Define a subgroup 
$\bT^{\bC}_{\Gamma'' \setminus \Gamma_0} 
 \cong (\bCx)^{2n-8}$ of 
$\bT_{\Gamma''}^{\bC} 
\cong (\bCx)^{\Gamma''}$ 
by
\begin{align}
  \bT^{\bC}_{\Gamma'' \setminus \Gamma_0}
  &= \{ (\tau_{\edg})_{\edg \in \Gamma''}
      \mid \tau_{\edg} = 1 
      \text{ for $\edg =  \edg(a,b), \edg(b,c), \edg(c,d), \edg(a,d)$} \} \\
  &\cong (\bCx)^{\Gamma'' \setminus 
    \{ \edg(a,b), \edg(b,c), \edg(c,d), \edg(a,d)\}},
\end{align}
and set
\begin{equation}
  \bT_{\Gamma_0}^{\bC} 
  = \bT_{\Gamma''}^{\bC} / 
    \bT_{\Gamma'' \setminus \Gamma_0}^{\bC}
  \cong (\bCx)^4.
\end{equation}

We consider an anti-canonical divisor of $X_0$ given by
\begin{equation}
  D_0 = \bigcup_{i=1}^n \{ p_{\lf{i}, \lf{i+1}} = 0 \}.
  \label{eq:anti-cacnonical_0}
\end{equation}
For each $\alpha = 0, \dots, n-4$, we define
\begin{equation}
  \widetilde{\Gr}^{\circ}(P_{\alpha})
  = \{ [p_{\edg, \edg'}^{P_{\alpha}}] \in \widetilde{\Gr}(P_{\alpha}) 
  \mid p_{\edg, \edg'}^{P_{\alpha}} \ne 0 
  \text{ for adjacent leaves $\edg, \edg'$} \}
\end{equation}
so that its projection image 
$\Gr^{\circ}(P_{\alpha})$ in $\Gr(P_{\alpha})$ 
is a complement of 
an anti-canonical divisor in $\Gr(P_{\alpha})$.
Since
$\widetilde{\Gr}^{\circ}(P_{\alpha}) \cong (\bCx)^3$
for $\alpha = 1, \dots, n-4$,
the torus 
$(\bCx)^{\amalg_{\alpha=0}^{n-4} \Gamma_{\alpha}}$
acts freely on 
$\prod_{\alpha=0}^{n-4} \widetilde{\Gr}^{\circ}(P_{\alpha})$.

\begin{proposition} \label{pr:GIT_complement}
The complement $X_0 \setminus D_0$ 
is isomorphic to the geometric quotient
\begin{equation}
  X_0 \setminus D_0 
  \cong \biggl. \prod_{\alpha=0}^{n-4} \widetilde{\Gr}^{\circ}(P_{\alpha})
  \biggr/ \Bigl(\bCx_{\partial P}
     \times \prod_{\edg \in \Int \Gamma''} \bCx_{\edg^+, \edg^-} \Bigr)
  \label{eq:quotient_complement}
\end{equation}
of $\prod_{\alpha} \widetilde{\Gr}^{\circ}(P_{\alpha})$.
\end{proposition}

\begin{proof}
First note that the image in $X_0 \subset \bP (\bigwedge^2 \bC^n)$
of the right hand side of 
\eqref{eq:quotient_complement} is the complement of 
a subvariety in $X_0$ defined by
$p_{\lf{i} \lf{j}} = 0$ for any $i < j$ such that 
$\gamma(\lf{i}, \lf{j})$ contains no path 
$\edg(a,b) \cup \edg(c,d)$, $\edg(a,d) \cup \edg(b,c)$
in $\Gamma_0$ connecting opposite sides of $P_0$.
Since 
each path $\gamma (\lf{i}, \lf{i+1})$ connecting
adjacent leaves contains neither  
$\edg(a,b) \cup \edg(c,d)$ nor $\edg(a,d) \cup \edg(b,c)$,
we have
$p_{\lf{i}, \lf{i+1}} = \prod p_{\edg, \edg'}^{P_{\alpha}} \ne 0$
for $i=1, \dots, n$,
which means that 
the complement $X_0 \setminus D_0$ contains 
the right hand side of \eqref{eq:quotient_complement}.
To show the converse,
we renumber $P_1, \dots , P_{n-4}$ if necessary
 in such a away that, for each $\alpha = 1, \dots, n-4$, 
the $\alpha$-th triangle 
$P_{\alpha}$ shares two sides with the polygon
$P \setminus (P_1 \cup \dots \cup P_{\alpha -1})$
with $n-(\alpha -1)$ sides.
Suppose that $\Gamma_1$ contains 
the $i$-th and $(i+1)$-st leaves of $\Gamma''$.
\begin{figure}[h]
  \centering
    \includegraphics[bb=0 0 97 82]{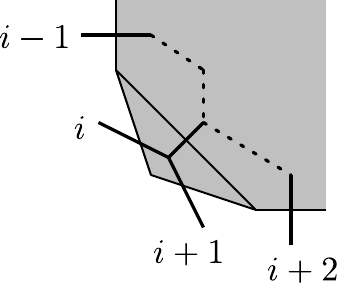}
  \caption{Paths connecting adjacent leaves}
  \label{fg:induction}
\end{figure}
Then the paths 
$\gamma(\lf{i-1}, \lf{i})$ and $\gamma(\lf{i+1}, \lf{i+2})$
share at least one interior edge 
(see Figure \ref{fg:induction}),
which implies from \eqref{eq:pluecker-binomial}
that the defining equation 
$F^{\Gamma''}_{\lf{i-1}, \lf{i}, \lf{i+1}, \lf{i+2}}(\bsp, \bszero) =0$
is a binomial
\begin{equation}
 p_{\lf{i-1}, \lf{i}} p_{\lf{i+1}, \lf{i+2}} 
  = p_{\lf{i-1}, \lf{i+1}} p_{\lf{i}, \lf{i+2}},
\end{equation}
and consequently we obtain
$p_{\lf{i-1}, \lf{i+1}},  p_{\lf{i}, \lf{i+2}} \ne 0$.
In other words, we  have $p_{\lf{j} \lf{k}} \ne 0$
for each pair $(\lf{j}, \lf{k})$ such that
the path $\gamma(\lf{j}, \lf{k})$ 
induces one in $\Gamma'' \setminus \Gamma_1$ 
connecting adjacent leaves in  
the polygon $P \setminus P_1$ with $n-1$ sides.
By repeating this process inductively, 
we obtain \eqref{eq:quotient_complement}.
\end{proof}

The $\bT^{\bC}_{\Gamma''}$-action on the quotient
$X_0 \setminus D_0$ is free, 
and the projection
$\prod_{\alpha=0}^{n-4} \widetilde{\Gr}^{\circ}(P_{\alpha})
  \to \widetilde{\Gr}^{\circ}(P_0)$ 
induces the quotient
\begin{equation}
  X_0 \setminus D_0 \longrightarrow 
  (X_0 \setminus D_0) / \bT_{\Gamma''}^{\bC}
  \cong \Gr^{\circ} (P_0) / \bT_{\Gamma_0}^{\bC}
  \cong \bCx,
\end{equation}
which extends to the GIT quotient
\begin{equation}
  X_0^{\mathrm{ss}} \longrightarrow X_0 \GIT \bT_{\Gamma''}^{\bC}
  \cong \Gr(P_0) \GIT \bT_{\Gamma_0}^{\bC}
  \cong \bP^1.
\label{eq:GIT_X_0}
\end{equation}
Note that the functions $\psi_{ac}^0$, $\psi_{bd}^0$ 
descend to bending Hamiltonians
$\varphi_{ac}$, $\varphi_{bd}$ (up to additive constants) 
on a 1-dimensional polygon space 
$\Gr(P_0) \GIT \bT_{\Gamma_0}^{\bC} \cong 
\Gr(P_0) \GIT_{\bsr} \bT_{\Gamma_0}$
parameterizing spatial quadrilaterals.
It follows from Proposition \ref{pr:GIT_complement}
that the GIT quotient of $X_0$ by the action of the subtorus
$\bT_{\Gamma'' \setminus \Gamma_0}^{\bC}$
induces a torus bundle
\begin{equation}
  X_0 \setminus D_0 \longrightarrow
  (X_0 \setminus D_0)
    / \bT_{\Gamma'' \setminus \Gamma_0}^{\bC}
  \cong \Gr^{\circ}(P_0)
  \label{eq:torus-bundle}
\end{equation}
over $\Gr^{\circ}(P_0)$.
The inclusion 
$\widetilde{\Gr}^{\circ}(P_0) \hookrightarrow 
\prod_{\alpha} \widetilde{\Gr}^{\circ}(P_{\alpha})$
defined by $p^{P_{\alpha}}_{\edg, \edg'} =1$
for all $\alpha \ne 0$ and $\edg, \edg'$
induces a section 
$\Gr^{\circ}(P_0) \to X_0 \setminus D_0$ of 
the torus bundle \eqref{eq:torus-bundle},
and thus we obtain the following.

\begin{corollary} \label{cr:complement}
The complement $X_0 \setminus D_0$ 
of the anti-canonical divisor $D_0$ is isomorphic 
to 
$\Gr^{\circ}(P_0) \times 
\bT_{\Gamma'' \setminus \Gamma_0}^{\bC}
\cong \Gr^{\circ}(2,4) \times (\bCx)^{2(n-4)}$.
\end{corollary}

\begin{remark}
From the argument in 
\cite[Section 5]{MR2609019},
the central fiber $X_0$ admits a small resolution 
$\pi \colon \widetilde{X}_0 \to X_0$ such that
$\widetilde{X}_0$ is a tower of projective planes over 
$\Gr(P_0)$.
The map $\pi$ is isomorphism on $X_0 \setminus D_0$, 
and the torus bundle structure \eqref{eq:torus-bundle}
is given by restricting the tower structure to the open subset
$X_0 \setminus D_0 \subset \widetilde{X}_0$.
\end{remark}

\section{Potential functions}


Let $(X, \omega)$ be a symplectic manifold, and 
fix a compatible almost complex structure $J$.
For a (relatively) spin 
Lagrangian submanifold $L$ 
the cohomology group $H^*(L; \Lambda_0)$
has a structure of a filtered $A_\infty$-algebra
\cite{MR2553465}
\begin{equation}
  \frakm_k \colon H^*(L; \Lambda_0)^{\otimes k}
  \longrightarrow H^*(L; \Lambda_0),
  \quad
  k = 0, 1, 2, \dots
\end{equation}
over the Novikov ring
\begin{align}
 \Lambda_0 = \lc
  \sum_{i=0}^\infty a_i T^{\lambda_i}
  \relmid a_i \in \bC, \ \lambda_i \in \bR_{\ge 0}, \ 
   \lim_{i \to \infty} \lambda_i = \infty \rc
\end{align}
defined by `counting' $J$-holomorphic disks
$(D^2 , \partial D^2) \to (X, L)$. 
A solution to the Maurer-Cartan equation
\begin{align}
 \sum_{k=0}^\infty \frakm_k(b, \dots, b) \equiv 0
  \mod \PD([L])
\end{align}
is called a weak bounding cochain,
where $\PD ([L])$ is the Poincar\'e dual
of the fundamental class $[L]$.
The potential function is a map
$
 \po : \scM(L) \to \Lambda_0
$
from the space $\scM(L)$ of weak bounding cochains 
defined by
\begin{equation}
 \sum_{k=0}^\infty \frakm_k(b, \dots, b) = \po(b) \cdot \PD([L]).
\end{equation}
For $b \in H^1(L; \sqrt{-1} \bR) \subset H^*(L; \Lambda_0)$
satisfying the Maurer-Cartan equation, 
the potential function is naively given by
\begin{align}
  \po (b) &= \sum_{\begin{subarray}{c}
   {\beta \in \pi_2(X,L),}\\
   {\mu_L(\beta) = 2}
   \end{subarray}}
  n_{\beta}(L) z_{\beta}(b),\\
  z_{\beta}(b) &= \hol_{b}(\partial \beta) 
  T^{\int_{\beta} \omega},
  \label{eq:z_beta}
\end{align}
where $\mu_L$ is the Maslov index, 
$\hol_{b}(\partial \beta)$ is the holonomy  of $b$
regarded as a flat $U(1)$-connection on $L$
 along the boundary $\partial \beta$, 
and $n_{\beta}(L)$ is the ``number'' of pseudo-holomorphic
disks in the class $\beta$ bounded by $L$
defined by
\begin{equation}
  \ev_* [\overline{\scM}_1(X, L; \beta)]
  = n(\beta) [L].
\end{equation}

Cho and Oh \cite{MR2282365}
and
Fukaya, Oh, Ohta and Ono \cite{MR2573826}
computed the potential functions for Lagrangian torus orbits 
in toric manifolds.
Combining this with Theorem \ref{th:toric_deg1},
one can compute the potential function
of Lagrangian torus fibers 
$L(\bsu) = L_{\Gamma}(\bsu)$ 
of the completely integrable system $\Psi_{\Gamma}$
on $\Gr(2,n)$.
Let 
\begin{align} \label{eq:ell_i}
 \ell_i(\bsu) = \langle \bsv_i, \bsu \rangle - \tau_i \ge 0
\end{align}
be the defining inequalities of $\Delta_{\Gamma}$
given in 
\eqref{eq:triangle_ineq1},
\eqref{eq:triangle_ineq2},
\eqref{eq:triangle_ineq3}$;$
\begin{align}
 \Delta_{\Gamma}
  = \{ \bsu \in \bR^{2n-4} \mid \ell_i(\bsu) \ge 0, \ i = 1, \dots, m \}.
\end{align}
Recall that
a holomorphic disk in the toric variety $X_{\Gamma}$ 
of Maslov index two bounded by a Lagrangian torus orbit 
intersect  transversally
a unique toric divisor at one point.
Let $\beta_i \in \pi_2(\Gr(2,n), L(\bsu))$ denote the class of 
a pseudo-holomorphic disk 
which is deformed into that in $X_{\Gamma}$ intersecting a toric divisor corresponding to the codimension one face 
$\{\ell_i(\bsu) = 0 \}$ of $\Delta_{\Gamma}$.

\begin{theorem}[{\cite[Theorem 8.1]{MR3211821}}] \label{th:potential}
For any $\bsu \in \Int \Delta_{\Gamma}$, 
one has an inclusion
$H^1(L(\bsu); \Lambda_0) \subset \mathcal{M}(L(\bsu))$,
and the potential function of $L(\bsu)$ is given by
\begin{align} 
 \po_{\Gamma} (L(\bsu), \bsx) 
  &= \sum_{i=1}^m z_{\beta_j}(\bsu, \bsx),
    \label{eq:po} \\
  z_{\beta_i}(\bsu, \bsx) 
  &= e^{\langle \bsv_i, x \rangle} T^{\ell_i(\bsu)}
    \label{eq:po2}
\end{align}
for 
$\bsx = (x_{ij})_{\edg(i,j) \in \Prn \Gamma} 
\in H^1(L(\bsu); \Lambda_0) 
\cong \Lambda_0^{2n-4}$.
\end{theorem}

By setting 
$y_{ij} = T^{u_{ij}}e^{x_{ij}}$ for  $\edg(i,j) \in \Prn \Gamma$
and $q = T^{\lambda}$,
we have a Laurent polynomial
\begin{equation}
  W_{\Gamma} \colon 
  (\bG_m)^{2n-4} \longrightarrow \bA^1
\end{equation}
in $\bsy = (y_{ij})_{\edg(i,j)\in \Prn \Gamma}$
defined by
\begin{align}
 \po_\Gamma(L(\bsu), \bsx)
  &= W_\Gamma(\bsy, q).
\end{align}
For $1 \le i < j \le n$, define a new variable $y(i,j)$
corresponding to $u(i,j)$ by
\begin{equation}
  y(i,j) = \begin{cases}
    y_{i, i+1}^{1/2}, & j=i+1 < n+1,\\
    q / ( \prod_{k=1}^{n-1} y_{k, k+1})^{1/2}, 
      & (i,j) = (1,n), \\ 
    y_{ij} / ( \prod_{k=i}^{j-1} y_{k, k+1})^{1/2} ,
      & |i-j| \ge 2.
  \end{cases}
\end{equation}
Then $W_{\Gamma}$ is given by
\begin{equation}
  W_{\Gamma} = \sum \left( \frac{y(i,j)y(j,k)}{y(i,k)} + 
  \frac{y(i,j)y(i,k)}{y(j,k)} + \frac{y(i,k)y(j,k)}{y(i,j)} \right),
\end{equation}
where the sum is taken over all triangles in the triangulation.

\begin{example}
Recall that
the polytope $\Delta_{\Gamma_{\cat}}$ 
corresponding to  the caterpillar $\Gamma_{\cat}$
is given by
\eqref{eq:GC}.
Then the potential function is given by
\begin{multline} \label{eq:po_caterpillar}
  W_{\Gamma_\cat} 
  = \frac{y_{1, 3}}{y_{1,2}} + \frac{y_{1,4}}{y_{1,3}} + \dots
  + \frac{y_{1, n-1}}{y_{1, n-2}} + \frac{q}{y_{1, n-1}} \\
  + \frac{y_{1, 2} y_{1, 3}}{\prod_{k=1}^2 y_{k, k+1}}
  + \frac{y_{1, 3} y_{1, 4}}{\prod_{k=1}^3 y_{k, k+1}} + \dots 
  + \frac{y_{1, n-2} y_{1, n-1}}{\prod_{k=1}^{n-2} y_{k, k+1}}
  + \frac{q y_{1, n-1}}{\prod_{k=1}^{n-1} y_{k, k+1}} \\
  + \frac{y_{1,2} y_{2,3}}{y_{1,3}}
  + \frac{y_{1,3} y_{3,4}}{y_{1,4}} + \dots 
  + \frac{y_{1, n-2} y_{n-2, n-1}}{y_{1, n-1}}
  + \frac{y_{1, n-1} y_{n-1, n}}{q} .
\end{multline}
\end{example}

\begin{proposition}  \label{pr:geom-lift}
Let $\Gamma$ and $\Gamma'$ be two triangulation
related by a Whitehead move in a quadrilateral with vertices $a,b, c,d$ 
$(1 \le  a <b < c < d \le n)$.
Then the corresponding potential functions 
$W_{\Gamma}$, $W_{\Gamma'}$ are related by
the geometric lift
\begin{equation}
  y_{ac}y_{bd} = 
  \frac{y_{ab} y_{bc} y_{cd} y_{ad}}
         { y_{ab} y_{cd} + y_{ad} y_{bc} / \prod_{i=b}^{c-1} y_{i, i+1}}
\label{eq:geom-lift}
\end{equation}
of the piecewise linear transformation \eqref{eq:PL-transform}
in the sense of \cite{Berenstein-Zelevinsky_TPM}.
\end{proposition}

\begin{proof}
Setting
$y_1 = y(a,b)$, $y_2 = y(b,c)$, $y_3 = y(c,d)$, $y_4 = y(a,d)$ and
$y = y(a,c)$, $y' = y(b,d)$, 
the potential functions corresponding 
to $\Gamma$ and $\Gamma'$ can be written as
\begin{align}
  W_{\Gamma} &= \frac{y_1y}{y_2} + \frac{y_2y}{y_1} + \frac{y_1y_2}{y} +
  \frac{y_3y}{y_4} + \frac{y_4y}{y_3} + \frac{y_3y_4}{y} + F(\bsy) \\
  &= \frac{(y_1y_3+y_2y_4)(y_1y_4+y_2y_3)}{y_1y_2y_3y_4} \cdot y
    + \frac{y_1y_2+y_3y_4}{y} + F(\bsy), \\
  W_{\Gamma'} 
  &= \frac{(y_1y_3+y_2y_4)(y_1y_2+y_3y_4)}{y_1y_2y_3y_4} \cdot y'
    + \frac{y_1y_4+y_2y_3}{y'} + F(\bsy),
\end{align}
for a Laurent polynomial $F(\bsy)$ independent of $y, y'$.
\eqref{eq:PL-transform2} implies that 
the coordinate change 
\eqref{eq:geom-lift} is given by
\begin{equation}
  y' = \frac 1y \cdot \frac{y_1y_2y_3y_4}{y_1y_3+y_2y_4},
\end{equation}
which transforms $W_{\gamma}$ into $W_{\Gamma'}$.
\end{proof}


\section{Cluster algebras}

The homogeneous coordinate ring $\bC[\Gr(2,n)]$
is generated by $\{ p_{ij} \}_{1 \le i < j \le n}$
with Pl\"ucker relations \eqref{eq:pluecker}.\footnote{%
The Grassmannian $\Gr(2,n) = \Gr(2, \bC^n)$ in this section is
canonically identified with its dual Grassmannian
$\Gr(n-2, (\bC^n)^*)$, which appears in the B-model side.
Hence the indices of the Pl\"ucker coordinates are
labels of vertices of $P$.}
It is a prototypical example
of a \emph{cluster algebra defined by a quiver},
also known as \emph{a skew-symmetric cluster algebra of geometric type}.
The notion of cluster algebras
is introduced in \cite{MR1887642}.
The cluster algebra structure on the homogeneous coordinate ring
of $\Gr(2,n)$ is established in \cite{MR2004457},
which is generalized to $\Gr(k, n)$ in \cite{MR2205721}.
It is also the cluster algebra
associated with a disk with $n$ marked points on the boundary,
which is a special case of  a cluster algebra
associated with a bordered surface with marked points
\cite{MR2233852,MR2130414,MR2448067}.

A \emph{quiver} $Q = (Q_0, Q_1, s, t)$
consists of a set $Q_0$ of \emph{vertices},
a set $Q_1$ of \emph{arrows}, and
two maps $s, t \colon Q_1 \to Q_0$
sending an arrow
to its \emph{source} and \emph{target} respectively.

A disk with $n$ marked points is homeomorphic
to the reference polygon $P$ with $n$ sides.
With each triangulation $\Gamma$ of the reference polygon $P$,
we associate a quiver $Q_\Gamma$
as shown in \pref{fg:triangulation2}.
\begin{figure}
\centering
\includegraphics[width=.5 \linewidth, bb=0 0 259 199]{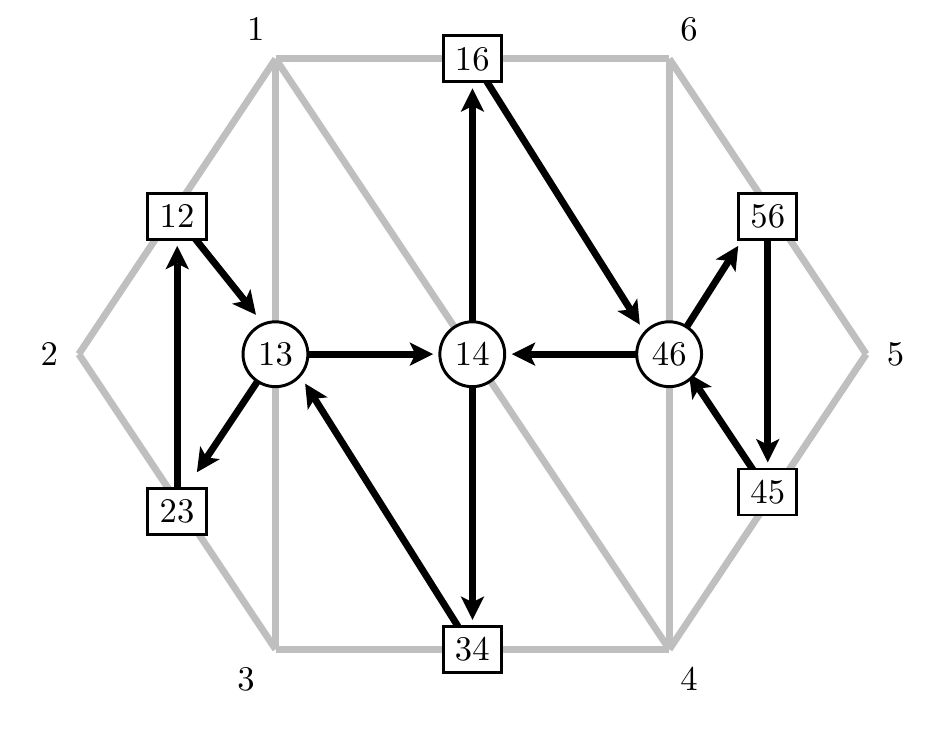}
\caption{A quiver}
\label{fg:triangulation2}
\end{figure}
The arrows in $Q_{\Gamma}$ are oriented 
in such a way that each cycle contained in a triangle 
of the triangulation is oriented clockwise.
Boxed vertices on the boundary
corresponding to edges of $P$ are \emph{frozen},
and circled vertices in the interior
corresponding to diagonals in $\Gamma$ are \emph{mutable}.
For a mutable vertex $v$ in the quiver $Q_\Gamma$,
the mutated quiver $\mu_v(Q_\Gamma)$ is constructed in three steps:
\begin{enumerate}
 \item
For each path $u \to v \to w$ of length two
passing through $v$,
add a new arrow $u \to w$.
 \item
Reverse all arrows with source or target $v$.
 \item
Annihilate pairs of arrows $a, b$
with $s(a) = t(b)$ and $t(a) = s(b)$,
in such a way that no oriented 2-cycle
(i.e., a path of length two with the same source and the target)
remains.
\end{enumerate}
One can easily see that
Whitehead moves
of triangulations correspond to
mutations of associated quivers.

We name mutable vertices as $v_1, \ldots, v_{n-3}$,
and frozen vertices as $v_{n-2}, \ldots, v_{2n-3}$.
With each vertex $v_i$,
we associate a variable $\rx_i$,
which is called a \emph{cluster variable} if $i=1, \ldots, n-3$ and
a \emph{frozen} (or \emph{coefficient}) \emph{variable}
if $i=n-2, \ldots, 2n-3$.
The sequence $\bfx = (\rx_1, \ldots, \rx_{n-3})$
is called a \emph{cluster}.
The pair $(\bfx, Q_\Gamma)$ of a cluster and a quiver is called
a \emph{labeled seed}.
Under the mutation $\mu_v$ of the quiver $Q_\Gamma$
at the vertex $v$,
the labeled seed is transformed as
$(\bfx, Q_\Gamma) \mapsto (\bfx', \mu_v(Q_\Gamma))$,
where $\rx_w' = \rx_w$ for $v \ne w$ and
\begin{align} \label{eq:mutation}
 \rx'_v \rx_v = \prod_{s(a) = v} \rx_{t(a)} + \prod_{t(a) = v} \rx_{s(a)}.
\end{align}
The \emph{cluster algebra}
is the $\bZ[\rx_{n-2}, \ldots, \rx_{2n-3}]$-subalgebra
of the \emph{ambient field}
$\bQ(\rx_1, \ldots, \rx_{2n-3})$
generated by cluster variables
in all the seeds
obtained from the initial seed
$(\bfx, Q_\Gamma)$ by any sequence of mutations.
One can easily see that
the cluster transformation \eqref{eq:mutation}
for the Whitehead move
interchanging the diagonals $d_{ik}$ and $d_{jl}$
gives exactly the Pl\"ucker relation \eqref{eq:pluecker}.
It follows that the cluster algebra in this case
is the homogeneous coordinate ring
of $\Gr(2,n)$.

\section{Landau--Ginzburg mirrors}

The mirror of $\Gr(2,n)$ is identified
with the Landau--Ginzburg model
\begin{align} \label{eq:MR2}
 \lb \Xv, \ 
 W = \sum_{i=1}^n \frac{p_{i,i+2}}{p_{i,i+1}} q^{\delta_{i,n-1}}
  \colon \Xv \to 
  \bA^1
  \rb
\end{align}
by Marsh and Rietsch
\cite{1307.1085},
where 
$\delta_{i, n-1}$ is the Kronecker delta, and
$\Xv := \Gr(2,n) \setminus D$
is the complement of an anti-canonical divisor
\begin{align}
 D = \{ p_{12}=0 \} \cup \{ p_{23}=0 \} \cup \cdots \cup \{ p_{n-1,n}=0 \}
  \cup \{ p_{1,n}=0 \}.
\end{align}
Here $(p_{ij})_{1 \le i<j \le n}$ is the Pl\"ucker coordinate on $\Gr(2,n)$
and $q = T^{\lambda}$
is an element of the quotient field $\Lambda$
of the Novikov ring $\Lambda_0$.
The arrow $W$ in \pref{eq:MR2} is a morphism of algebraic varieties
over $\Lambda$.
The Landau--Gizburg model \pref{eq:MR2} is a special case of \cite{MR2397456},
where Landau--Ginzburg mirrors
of general flag varieties are introduced.

An open subspace of this Landau--Ginzburg model
is given earlier in
\cite{MR1439892,
MR1619529,
MR1756568}:
Consider a quiver $Q = (Q_0, Q_1, s, t)$ of the form
\begin{equation} \label{eq:ladder}
\begin{alignedat}{17}
  \ldbox{q} &&&& \ldbox{\dfrac{p_{1,n}}{p_{12}}} \\
  & \nwarrow && \swarrow && \nwarrow &&&&&&&&&&& \\
  && \ldbox{\dfrac{p_{n-1,n}}{p_{1,n-1}}} &&&& 
       \ldbox{\dfrac{p_{1,n-1}}{p_{12}}} \\
  &&& \nwarrow && \swarrow && \nwarrow &&&&&&&&& \\
  &&&& \ldbox{\dfrac{p_{n-2,n-1}}{p_{1,n-2}}} &&&& 
             \ldbox{\dfrac{p_{1,n-2}}{p_{12}}} \\
  &&&&& \nwarrow && \swarrow && \nwarrow &&&&&&& \\
  &&&&&& \ldbox{\dndots} &&&& \ldbox{\dndots} &&&& 
                \ldbox{1} && \\
  &&&&&&& \nwarrow && \swarrow && \nwarrow && \swarrow \\
  &&&&&&&& \ldbox{\dfrac{p_{34}}{p_{13}}} &&&& 
                     \ldbox{\dfrac{p_{13}}{p_{12}}} \\
  &&&&&&&&& \nwarrow && \swarrow &&&&&& \\
  &&&&&&&&&& \ldbox{\dfrac{p_{23}}{p_{12}}} \\
\end{alignedat},
\end{equation}
where vertices are Laurent monomials 
in the Pl\"ucker coordinate.
It is shown in \cite[Proposition 5.9]{1307.1085}
that the restriction of the Landau--Ginzburg potential
\eqref{eq:MR2}
to the torus in $\Gr(2,n)$ defined by
$p_{i,i+1} \ne 0$ for $i=1, \ldots, n-1$ and
$p_{1,i} \ne 0$ for $i = 1, \ldots, n-1$
is given by
\begin{align} \label{eq:BCKv}
 W = \sum_{a \in Q_1} \frac{t(a)}{s(a)}.
\end{align}

For each triangulation $\Gamma$ of the reference polygon,
define an open embedding 
$\iota_{\Gamma} \colon  U_{\Gamma} \coloneqq (\bG_m)^{\Prn \Gamma}
\hookrightarrow \check{X}$
by
\begin{equation}
  y_{ij} = 
  \begin{cases}
    \displaystyle{ q \frac{p_{1n}}{p_{in}},} \quad 
      & \text{$i=2, 3, \dots, n-1$ and $j=n$,}\\[10pt]
    \displaystyle{\frac{p_{j, j+1}}{p_{ij}},} 
      & \text{otherwise}.
  \end{cases}
  \label{eq:emb_U}
\end{equation}

\begin{remark}
Applying \eqref{eq:emb_U} formally
to the case $(i,j) = (1,n)$, we obtain
$y_{1n} = q$, which is consistent with 
the fact that $\psi_{1n} = \lambda$ is constant.
\end{remark}

\begin{theorem} \label{th:cluster_transf}
\begin{enumerate}
\item
For each triangulation $\Gamma$ of the reference polygon,
the potential function $W_{\Gamma}$ is the restriction 
of the Marsh-Rietsch superpotential \eqref{eq:MR2};
\begin{equation}
  W_{\Gamma} = \iota_{\Gamma}^* W.
\end{equation}
\item
Let $\Gamma$ and $\Gamma'$ be two triangulation
related by a Whitehead move in a quadrilateral with 
vertices $a,b, c, d$   $(1 \le  a <b < c < d \le n)$.
Then the transformation \eqref{eq:geom-lift} is equivalent
to the Pl\"ucker relation
\begin{equation}
  p_{ac}p_{bd} = p_{ab}p_{cd} + p_{ad}p_{bc}.
  \label{eq:pluecker_rel}
\end{equation}
under the coordinate change \eqref{eq:emb_U}.
\end{enumerate}
\end{theorem}

\begin{proof}
In the case of caterpillar $\Gamma_{\cat}$, 
it is straightforward to see from 
\eqref{eq:po_caterpillar} (or \eqref{eq:GC}) and \eqref{eq:ladder} 
that the Landau--Ginzburg potential
\eqref{eq:BCKv}
is identified with the potential function
\eqref{eq:po}
under the coordinate change
\eqref{eq:emb_U}.
One can also easily check that 
the coordinate change \eqref{eq:geom-lift}, 
which can be written as 
\begin{equation}
  \frac{1}{y_{ac}y_{bd}} 
  = \frac{1}{y_{ab} y_{cd} \prod_{i=b}^{c-1} y_{i,i+1}}
  + \frac{1}{y_{ad}y_{bc}},
\end{equation}  
is equivalent to the Pl\"ucker relation \eqref{eq:pluecker_rel}.
This implies that $U_{\Gamma}$ and $U_{\Gamma'}$ are
glued together in $\check{X} \subset \Gr(2,n)$.
Since any triangulation is related to the caterpillar
by a sequence of Whitehead moves,
Proposition \ref{pr:geom-lift} prove the first statement 
for any $\Gamma$.
\end{proof}

\begin{remark}
The union $\bigcup_{\Gamma} U_{\Gamma}$ 
does not cover the whole $\check{X}$ in general.
In the case of $\Gr(2,4)$, 
the complement of the open subset
$\bigcup_{\Gamma} U_{\Gamma}$
is given by
\begin{equation}
  \check{X} \setminus \bigcup_{\Gamma} U_{\Gamma}
  = \{ [p_{ij}] \in \check{X} \mid p_{13} = p_{24} = 0 \}
  \cong (\bG_m)^2.
\end{equation}
In this case,
the superpotential $W$ has two critical points 
with zero critical value, which are contained in 
this complement.
We expect that singular Lagrangian fibers
of the Lagrangian torus fibration
interpolating $\Psi_\Gamma$ for different $\Gamma$
correspond to points in this complement
under homological mirror symmetry
just as in \cite[Section 8]{MR3439224}.
See also \cite{MR3601889,1704.07213,1801.07554}.
\end{remark}

\section{Wall-crossing formula}

Let $\Phi \colon X \to B$ be a Lagrangian torus fibration
on a symplectic manifold $(X, \omega)$ 
possibly with singular fibers.

\begin{definition}
A Lagrangian fiber $L(\bsu) = \Phi^{-1}(\bsu)$ is said to be
\emph{potentially obstructed} if it bounds a pseudo-holomorphic 
disk of Maslov index zero.
The set of $\bsu \in B$ with potentially obstructed fiber $L(\bsu)$ 
is called a \emph{wall}.
\end{definition}

We assume that any wall has codimension one, and 
any Lagrangian fiber $L(\bsu)$ satisfies the 
following conditions(\cite[Assumptions 3.2,  3.8]{MR2386535}):
\begin{itemize}
\item 
there are no non-constant holomorphic sphere $v \colon \bP^1 \to X$ with 
$c_1(TX)\cdot [v] \le 0$;
\item
holomorphic disks of Maslov index two in $(X, L(\bsu))$ are regular;
\item
all simple (non multiply covered) non-constant holomorphic disks 
in $(X, L(\bsu))$ of Maslov index zero are regular,
and the associated evaluation maps 
at boundary marked points
are transverse to each other
and to the evaluation maps
at boundary marked points
of holomorphic disks
of Maslov index 2.
\end{itemize}
Let $U_0, U_1 \subset B$ be two chambers separated by a wall
on which each fiber bounds a unique non-constant 
pseudo-holomorphic disk
of Maslov index zero.
Let $\alpha \in \pi_2(X, L(\bsu))$ be the class of such a disk.

\begin{proposition}[Auroux {\cite[Proposition 3.9]{MR2386535}}]
For $\bsu_0 \in U_0$ and $\bsu_1 \in U_1$, 
the functions $z_{\beta}$ given in \eqref{eq:z_beta} 
defined for $L(\bsu_0)$ and $L(\bsu_1)$ are related by
\begin{equation}
  z_{\beta} \longmapsto z_{\beta} 
  h(z_{\alpha})^{[\partial \alpha] \cdot [\partial \beta]}
  \label{eq:wallcrossing_formula}
\end{equation}
for a function $h(z_{\alpha}) = 1 + O(z_{\alpha})$. 
\end{proposition}

This wall-crossing formula also follows from the isomorphism of filtered $A_\infty$-algebras
associated with the change of almost complex structures
given in \cite{MR2553465}.

Now we recall an example of the wall-crossing formula 
given by Auroux in
\cite[Section 5]{MR2386535} and 
\cite[Section 3.1]{MR2537081}.

\begin{example} \label{eg:Auroux}
Consider the $\bCx$-action on
\begin{equation}
  Y = \{ (x_1, x_2, x_3) \in \bC^3 \mid x_1 x_2 = 1+ x_3 \} 
  \cong \bC^2 = \bC_{x_1} \times \bC_{x_2}
\label{eq:Y}
\end{equation}
defined by
\begin{equation}
  \tau \cdot (x_1,x_2,x_3) = (\tau x_1, \tau^{-1}x_2, x_3), \quad
  \tau \in \bCx.
\end{equation}
Then the projection 
$f \colon Y \to \bC$, $(x_1,x_2,x_3) \mapsto x_3$
gives the GIT quotient with respect to the $\bCx$-action.
Equip $Y$ with an $S^1 (\subset \bCx)$-invariant
symplectic form, and let $\mu_{S^1} \colon Y \to \bR$
be the moment map of the $S^1$-action.
Define a Lagrangian torus fibration on $Y$ by
\begin{equation}
  \Phi \colon Y \to \bR_{\ge 0} \times \bR,
  \quad
  x= (x_1,x_2,x_3) \mapsto \lb \abs{x_3}, \mu_{S^1}(x) \rb,
\end{equation}
and write its fibers as
\begin{equation}
  T_{R,r} = \Phi^{-1}(R,r) =
  \{ x \in Y \mid 
  |x_3| = R, \,
  \mu_{S^1}(x) = r,
  \}.
  \label{eq:T_rR}
\end{equation}
This fibration has a unique singular fiber $T_{1,0}$ over $(1,0)$,
which is a two-torus with one circle pinched.
For $r \ne 0$, each fiber $T_{1,r}$ over the line $R=1$
intersects a coordinate axis of $Y$ viewed as  
$\bC_{x_1} \times \bC_{x_2}$ at a circle, 
and thus 
it bounds a holomorphic disk  of Maslov index zero, 
which has the form 
$D^2 \times \{0\}$ or $\{0\} \times D^2$
in  $\bC_{x_1} \times \bC_{x_2} \cong Y$.
Hence the wall of $\Phi$  is given by $R=1$.
Let $\alpha \in \pi_2(Y, T_{1,r})$ denote the class of 
disks of the form $D^2 \times \{0\}$.
Lagrangian fibers over the chambers $R > 1$ and $R<1$
are said to be of \emph{Clifford type} and
of \emph{Chekanov type}, respectively.
A Clifford type fiber $T_{r, R}$ 
can be deformed into a torus of the form
$S^1(r_1) \times S^1(r_2) \subset \bC_{x_1} \times \bC_{x_2}$,
which bounds holomorphic disks of Maslov index two
of the forms 
$D^2(r_1) \times \{ \mathrm{pt} \}$ 
and $\{ \mathrm{pt} \} \times D^2(r_2)$.
Let $\beta_1, \beta_2 \in \pi_2(Y, T_{r,R})$ be the classes
of these holomorphic disks.
On the other hand, a Chekanov type torus $T_{r, R}$ 
bounds a family of 
holomorphic disks of Maslov index two, 
which are
sections of $f \colon Y\to \bC$ over the disk $D^2(R)$
enclosed by the image 
$f(T_{r,R}) = \{ x_3 \in \bC \mid |x_3| = R \}$
of $T_{r,R}$.
Let $\beta_3$ denote its homotopy class.
Then the wall-crossing formula is given by
\begin{equation}
  z_{\beta_3} =
  \begin{cases}
    z_{\beta_2}(1+z_{\alpha}) \quad 
   & (r>0),\\
   z_{\beta_1}(1+z_{\alpha}^{-1})
   \quad 
  & (r<0).
\end{cases}
\label{eq:Auroux_wallcrossing}
\end{equation}
Since $z_{\alpha} = z_{\beta_1}/z_{\beta_2}$, 
the transformations \eqref{eq:Auroux_wallcrossing} 
on $r>0$ and $r<0$ are identical.
\end{example}

Let us go back to the case of the Grassmannian $\Gr(2,n)$.
Suppose we have two triangulation $\Gamma, \Gamma'$ of 
the reference polygon related by a Whitehead move
in a quadrilateral with vertices
$1\le a < b < c < d \le n$,
and let $\Gamma''$ be the subdivision given by common diagonals
in $\Gamma$ and $\Gamma'$.
Note that $\Psi_{\Gamma}$ and $\Psi_{\Gamma'}$ are written as
\begin{equation}
  \Psi_{\Gamma} =
  ((\psi_{ij})_{\edg(i,j) \in \Prn \Gamma''}, \psi_{ac}),
  \quad
  \Psi_{\Gamma'} =
  ((\psi_{ij})_{\edg(i,j) \in \Prn \Gamma''}, \psi_{bd}).
\end{equation}
Since 
$\varphi_{ac} = \psi_{ac} - \frac 12 \sum_{i=a}^{c-1} \psi_{i, i+1}$ 
and 
$\varphi_{bd} = \psi_{bd} - \frac 12 \sum_{i=b}^{d-1} \psi_{i, i+1}$ 
Poisson commute with
all other $\psi_{ij}$, $\edg(i,j) \in \Prn \Gamma''$, 
the function
\begin{equation}
  \psi_t = (1-t) (\varphi_{ac})^2 - t (\varphi_{bd})^2
\end{equation}
also Poisson commutates with the functions $\psi_{ij}$
for any $t \in [0,1]$:
\begin{equation}
  \{ \psi_t, \psi_{ij} \} = 0.
\end{equation}
Hence we obtain the following.

\begin{proposition}
Let $\Gamma$, $\Gamma'$ be two triangulations of $P$ 
as above.
Then 
\begin{equation}
  \Psi_t = \bigl( 
  (\psi_{ij})_{\edg(i,j) \in \Prn \Gamma''},
  \psi_{t} \bigr),
  \quad t \in [0,1]
\end{equation}
is a one-parameter family of completely integrable systems 
on $\Gr(2,n)$ connecting $\Psi_{\Gamma}$ and 
$\Psi_{\Gamma'}$.
\end{proposition}

\begin{theorem} \label{th:wall-crossing}
Let $\Gamma$ and $\Gamma'$ be two triangulations as above.
For $t \in (0,1)$, the wall-crossing formula 
\eqref{eq:wallcrossing_formula} for $\Psi_t$ is equivalent to
the coordinate change \eqref{eq:geom-lift} 
(and hence, to the Pl\"ucker relation \eqref{eq:pluecker_rel}).
\end{theorem}

We prove this theorem in the next two sections.

\section{Wall-crossing formula on $\Gr(2,4)$}

In this section we prove Theorem \ref{th:wall-crossing}
in the case of $\Gr(2,4)$.

Let $\Gamma, \Gamma'$ be the two triangulations
of a quadrilateral $P$ given by diagonals $d_{13}$, $d_{24}$, 
respectively.
Then the corresponding potential functions 
\begin{align}
  W_{\Gamma} &= \frac{y_{13}}{y_{12}} + \frac{y_{13}}{y_{23}}
    + \frac{y_{12}y_{23}}{y_{13}} + \frac{qy_{13}}{y_{12}y_{23}y_{34}}
    + \frac{y_{13}y_{34}}{q} + \frac{q}{y_{13}}, \\
  W_{\Gamma'} &= \frac{y_{24}}{y_{23}} + \frac{y_{24}}{y_{23}}
    + \frac{y_{23}y_{34}}{y_{24}} + \frac{qy_{24}}{y_{12}y_{23}y_{34}}
    + \frac{y_{12}y_{24}}{q} + \frac{q}{y_{24}}
\end{align}
are related by the coordinate change
\eqref{eq:geom-lift}, which is given by 
\begin{equation}
  \frac{1}{y_{13}y_{24}} = \frac{1}{q y_{23}}
  \left( 1 + \frac{q}{y_{12}y_{34}} \right)
\label{eq:cluster_transf}
\end{equation}
in this case.

Recall that the image of the moment map
$\mu_{\bT_{U(4)}} \colon \Gr(2,4) \cong \scO_{\lambda} 
  \to \bR^4$
is an octahedron in 
$\{ \bsr = (r_1, \dots, r_4)  \in \bR^4 \mid 
\sum_{i=1}^4 r_i = \lambda \} \cong \bR^3$
defined by
\begin{equation}
  0 \le r_i \le  
  \sum_{j \ne i} r_j , \qquad 
  i = 1, \dots, 4,
\end{equation}
shown in \pref{fg:muTU4},
and critical values of $\mu_{\bT_{U(4)}}$ inside 
the octahedron $\mu_{\bT_{U(4)}}(\scO_{\lambda})$ form three walls
$H_1 \cup H_2 \cup H_3$, where
\begin{align}
  H_1 &= \{\bsr \mid r_1 + r_2 = r_3 + r_4 \} , \\
  H_2 &= \{\bsr \mid r_1 + r_3 = r_2 + r_4 \} , \\
  H_3 &= \{\bsr \mid r_1 + r_4 = r_2 + r_3 \}
\end{align}
(see \cite[Proposition 4.3]{Hausmann-Knutson_PSG}
or \eqref{eq:discriminant_G} below).
\begin{figure}[ht]
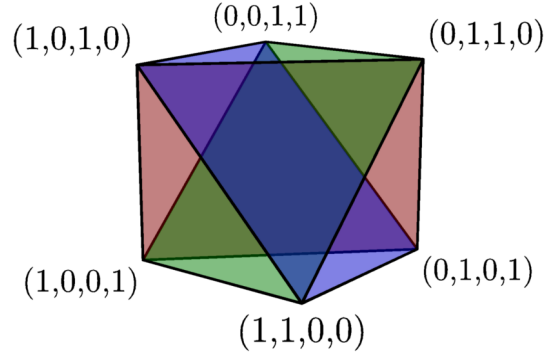

\centering
\begin{asy}
import three;
import polyhedron_js;

size(7cm,0);

polyhedron[] poly1;
poly1[0]=(0,0,1)--(1,0,0)--(1,0,1)--cycle;
poly1[1]=(0,0,1)--(1,0,1)--(0,1,1)--cycle;
poly1[2]=(0,0,1)--(0,1,1)--(0,1,0)--cycle;
poly1[3]=(0,0,1)--(0,1,0)--(1,0,0)--cycle;
poly1[4]=(1,1,0)--(1,0,0)--(1,0,1)--cycle;
poly1[5]=(1,1,0)--(1,0,1)--(0,1,1)--cycle;
poly1[6]=(1,1,0)--(0,1,1)--(0,1,0)--cycle;
poly1[7]=(1,1,0)--(0,1,0)--(1,0,0)--cycle;

path3 h1=(1,0,0)--(1,0,1)--(0,1,1)--(0,1,0)--cycle;
path3 h2=(1,0,0)--(1,1,0)--(0,1,1)--(0,0,1)--cycle;
path3 h3=(0,1,0)--(1,1,0)--(1,0,1)--(0,0,1)--cycle;

filldraw(poly1,new pen[]{gray},op=0);
draw(surface(h1),red+opacity(0.5));
draw(surface(h2),green+opacity(0.5));
draw(surface(h3),blue+opacity(0.5));

label("(0,0,1,1)", (0,0,1),N);
label("(1,0,0,1)", (1,0,0),SW);
label("(1,0,1,0)", (1,0,1),NW);
label("(0,1,1,0)", (0,1,1),NE);
label("(0,1,0,1)", (0,1,0),SE);
label("(1,1,0,0)", (1,1,0),S);

\end{asy}
\caption{The moment polytope of $\mu_{\bT_{U(4)}}$
and the hyperplanes $H_1$ (red),
$H_2$ (green), and $H_3$ (blue)
(rotatable with Acrobat Reader)}
\label{fg:muTU4}
\end{figure}
For each interior point 
$\bsr$ in $\mu_{\bT_{U(4)}}(\scO_{\lambda})$,
the completely integrable system
\begin{equation}
  \Psi_t = ( \psi_{12}, \psi_{23}, \psi_{34}, 
  (1-t)(\varphi_{13})^2 - t (\varphi_{24})^2)
\end{equation}
induces the function
\begin{equation}
  \varphi_t = (1-t)(\varphi_{13})^2 - t (\varphi_{24})^2
  \colon \scM_{\bsr} = \mu_{\bT_{U(4)}}^{-1}(\bsr) / \bT_{U(4)}
  \longrightarrow \bR
\end{equation}
on the polygon space.
Let $B_t = \Psi_t(\Gr(2,4))$ and 
$I_{t, \bsr} = \varphi_t(\scM_{\bsr})$ denote the ranges 
of $\Psi_t$ and $\varphi_t$, respectively.
Since $(\psi_{12}, \psi_{23}, \psi_{34})$ gives the moment map
of the $\bT_{U(4)}$-action on $\scO_{\lambda}$,
we have a natural projection
\begin{equation}
 B_t \longrightarrow \mu_{\bT_{U(4)}}(\scO_{\lambda})
 \subset \bR^3,
 \quad
 (u_1,u_2, u_3, u_4) \longmapsto (u_1, u_2, u_3).
\end{equation}

\begin{lemma} \label{lm:deform_bending}
For each $t \in (0, 1)$ 
and $\bsr \in \Int \mu_{\bT_{U(n)}}(\scO_{\lambda})$,
the map 
$\varphi_t \colon \scM_{\bsr} \to I_{t, \bsr}$ 
is an $S^1$-bundle over the interior $\Int I_{t, \bsr}$, 
and the fibers over boundary points of $I_{t, \bsr}$ are single points.
\end{lemma}

\begin{proof}
We first note that the $S^1$-fibration 
$\varphi_{13} \colon \scM_{\bsr} \to I_{0, \bsr}$
(resp. $\varphi_{24}$) is not homeomorphic to 
the toric moment map on $\bP^1$ 
exactly when $\min \varphi_{13} = 0$
(resp. $\min \varphi_{24} = 0$),
or equivalently, $r_1=r_2$ and $r_3=r_4$
(resp. $r_1=r_4$ and $r_2=r_3$);
in this case, the fiber $\varphi_{13}^{-1}(0)$ consists of 
``broken lines'' 
$\bsxi = ( \xi_1, \dots, \xi_4)$ satisfying 
$\xi_1 + \xi_2 = \xi_3 +\xi_4 =0$,
and it is diffeomorphic to a line segment.
We consider the map 
\begin{equation}
   \varphi = ( (\varphi_{13})^2, (\varphi_{24})^2) \colon
  \scM_{\bsr} \longrightarrow \bR^2,
\end{equation}
and let $(v_1, v_2)$ be the standard coordinate on $\bR^2$.
Then the boundary of the image $\varphi (\scM_{\bsr})$
contains a line segment in a coordinate axis
if $\min \varphi_{13} = 0$ or $\min \varphi_{24} = 0$
(see Figure \ref{fg:bending_image3} and 
Figure \ref{fg:bending_image5}).
\begin{figure}[h]
  \centering
  \begin{minipage}{6cm}
    \centering
    \includegraphics[bb=0 0 79 75]{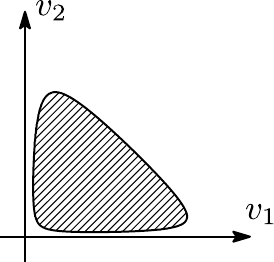}
    \caption{$\varphi (\scM_{\bsr})$ for generic $\bsr$}
    \label{fg:bending_image1}
  \end{minipage}
  \hspace{1cm}
  \begin{minipage}{5cm}
    \centering
    \includegraphics[bb=0 0 79 75]{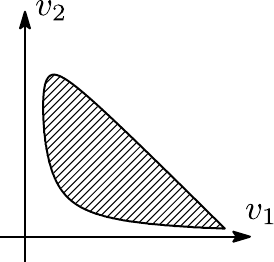}
    \caption{$\varphi (\scM_{\bsr})$ for $\bsr \in H_1$}
    \label{fg:bending_image2}
  \end{minipage}
  \\[5mm]
  \begin{minipage}{5cm}
    \centering
    \includegraphics[bb=0 0 78 75]{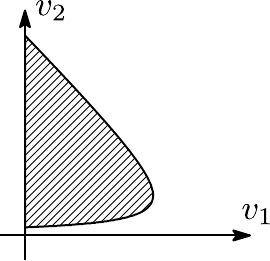}
    \caption{$\varphi (\scM_{\bsr})$  in the case of $\bsr \in H_2 \cap H_3$}
    \label{fg:bending_image3}
  \end{minipage}
  \hspace{1cm}
  \begin{minipage}{5cm}
    \centering
    \includegraphics[bb=0 0 79 75]{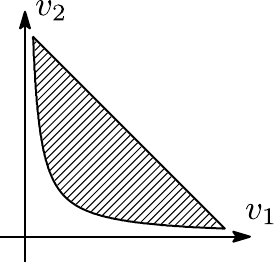}
    \caption{$\varphi (\scM_{\bsr})$  in the case of $\bsr \in H_1 \cap H_3$}
    \label{fg:bending_image4}
  \end{minipage}
  \\[5mm]
    \includegraphics[bb=0 0 78 75]{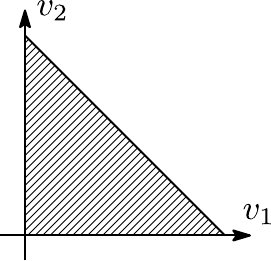}
    \caption{$\varphi (\scM_{\bsr})$ in the case of 
                $\bsr \in H_1 \cap H_2 \cap H_3$}
    \label{fg:bending_image5}
\end{figure}

\begin{claim} \label{cl:convex}
For any $\bsr$ and $t \in (0,1)$, the intersection 
\begin{equation}
 \varphi(\scM_{\bsr}) \cap
 \{ (v_1, v_2) \in \bR^2 \mid (1-t) v_1 - t v_2 = c \} 
\end{equation}
is a line segment for $c \in \Int I_{t, \bsr}$, and
is a single point in $\partial \varphi(\scM_{\bsr})$
when $c \in \partial I_{t, \bsr}$.
\end{claim}

Assuming Claim \ref{cl:convex},
Lemma \ref{lm:deform_bending} follows from the fact that
$\varphi \colon \scM_{\bsr} \to \varphi(\scM_{\bsr})$
is a double cover which branches along the boundary 
$\partial \varphi(\scM_{\bsr})$.
\end{proof} 

\begin{proof}[Proof of Claim \ref{cl:convex}.]
Fix $v_1 > 0$ and consider a quadrilateral $\bsxi \in \scM_{\bsr}$
with $\varphi_{13}(\bsxi) = \sqrt{v_1}$.
\begin{figure}[h]
  \centering
    \includegraphics[bb=0 0 92 79]{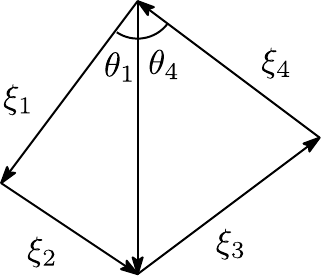}
  \caption{A diagonal $\xi_1 + \xi_2$ in a quadrilateral
              $\bsxi \in \scM_{\br}$.}
  \label{fg:quadrilateral1}
\end{figure}
Then the angles $\theta_1$ (resp. $\theta_4$)  
between the side $\xi_1$ (resp.  $-\xi_4$) 
and the diagonal $\xi_1+\xi_2$
connecting the first and the third vertices
are given by
\begin{equation}
  \cos \theta_1 = \frac{v_1 + r_1^2 - r_2^2}{2r_1\sqrt{v_1}},
  \quad
  \cos \theta_4 = \frac{v_1 + r_4^2 - r_3^2}{2r_4 \sqrt{v_1}}
  \label{eq:cos}
\end{equation}
(see \pref{fg:quadrilateral1}).
Since $\varphi_{24}$ restricted to the level set $\varphi_{13}^{-1}(\sqrt{v_1})$
takes its maximum and minimum when 
$\bsxi$ is contained in a plane,
the range of $(\varphi_{24})^2 |_{\varphi_{13}^{-1}(\sqrt{v_1})}$ is 
\begin{multline} \label{eq:sin_cos}
  r_1^2+r_4^2-2r_1r_4
  (\cos \theta_1 \cos \theta_4 + \sin \theta_1 \sin \theta_4)
  \\
  \le v_2 \le 
  r_1^2+r_4^2-2r_1r_4
  (\cos \theta_1 \cos \theta_4 - \sin \theta_1 \sin \theta_4).
\end{multline}
Equality in \pref{eq:sin_cos} holds
if and only if
\begin{align}
 \cos \theta_1 \cos \theta_4 \pm \sin \theta_1 \sin \theta_4 
 = \frac{r_1^2+r_4^2-v_2}{2 r_1 r_4},
\end{align}
which, combined with \eqref{eq:cos}, gives
\begin{align} \label{eq:sin}
 \pm \sin \theta_1 \sin \theta_4 = \frac{r_1^2+r_4^2-v_2}{2 r_1 r_4}
  - \frac{v_1 + r_1^2 - r_2^2}{2 r_1 \sqrt{v_1}}
  \cdot \frac{v_1 + r_4^2 - r_3^2}{2 r_4 \sqrt{v_1}}.
\end{align}
By taking the square of the both sides of \pref{eq:sin}
and using
\begin{align}
 \sin^2 \theta_1 \sin^2 \theta_4
  &= \lb 1 - \cos^2 \theta_1 \rb \lb 1 - \cos^2 \theta_4 \rb,
\end{align}
we see that
\begin{multline}
  F(v_1, v_2) \coloneqq
  v_1^2v_2 + v_1v_2^2 - (r_1^2+r_2^2+r_3^2+r_4^2)v_1v_2 \\
  + (r_1^2-r_4^2)(r_2^2-r_3^2)v_1
  + (r_1^2-r_2^2)(r_4^2-r_3^2)v_2 \\
  + (r_1^2-r_2^2+r_3^2-r_4^2)(r_1^2r_3^2-r_2^2r_4^2)
\end{multline}
gives the defining equation 
for the boundary of the image 
$\varphi(\scM_{\bsr}) \subset \bR^2$.
The discriminant of $F(v_1,v_2)$ is given by
\begin{multline}
 r_1^2 r_2^2 r_3^2 r_4^2
 (r_1+r_2+r_3+r_4)^2
 (-r_1+r_2+r_3+r_4)^2 \\
 (r_1-r_2+r_3+r_4)^2
 (r_1+r_2-r_3+r_4)^2 \\
 (r_1+r_2+r_3-r_4)^2
 (r_1+r_2-r_3-r_4)^2 \\
 (r_1-r_2+r_3-r_4)^2
 (r_1-r_2-r_3+r_4)^2,
\label{eq:discriminant_G}
\end{multline}
which means that the discriminant locus is 
$\partial (\mu_{\bT_{U(4)}}(\scO_{\lambda}))
\cup H_1 \cup H_2 \cup H_3$.

When $\bsr$ does not lie
on the discriminant locus,
then $\partial (\varphi(\scM_{\br}))$
is the positive real part
of a smooth plane cubic curve.
A smooth real plane cubic curve has either
\begin{itemize}
\item
one non-compact connected component, or
\item
one non-compact connected component
and one compact connected component.
\end{itemize}
Only the latter can happen in our case,
and the boundary $\partial (\varphi(\cM_{\bsr}))$
of $\varphi(\cM_{\bsr})$ is the compact connected component
of the cubic curve.
When $\bsr$ lies on exactly one wall,
say $H_1$,
then the resulting cubic curve has one node,
and $\varphi(\cM_{\bsr})$ is the closure
of the compact connected component
of the complement of the nodal cubic curve
(see Figure \ref{fg:bending_image2}).
In either case, $\varphi(\scM_{\br})$ is strictly convex.

Next we consider the case where $\bsr$ lies on exactly two walls.
If $\br$ lies in  $H_2$ and $H_3$,
one has $r_1 = r_2 \ne r_3 = r_4$,
so that
\begin{align}
 F(v_1, v_2)
  &= v_1 \lb (r_1^2 - r_3^2)^2 + v_1 v_2 - 2 (r_1^2 + r_3^2) v_2 + v_2^2 \rb
\end{align}
and $\varphi(\cM_{\bsr})$ is bounded
by a hyperbola and the $v_2$-axis.
Similarly, if $\br$ lies in $H_1 \cap H_2$, 
the image $\varphi (\scM_{\bsr})$ is bounded by a hyperbola and 
the $v_1$-axis.
When $\br$ lies in $H_1 \cap H_3$, 
one has $r_1 = r_3 \ne r_2 = r_4$, and thus 
\begin{equation}
F(v_1, v_2) = (v_1 + v_2 - 2(r_1^2 + r_2^2))(v_1 v_2 - (r_1^2 - r_2^2)^2),
\end{equation}
which means that $\varphi(\scM_{\br})$ is bounded by a hyperbola
and a line as in Figure \ref{fg:bending_image4}.

When $\br$ lies on all the three walls $H_1$, $H_2$ and $H_3$,
then one has
$r_1 = r_2 = r_3 = r_4$,
so that
\begin{align}
 F(v_1, v_2)
  &= v_1 v_2 ( v_1 + v_2 - 4 r_1^2 ),
\end{align}
and $\varphi(\cM_{\bsr})$ is a triangle.

In all these cases,
the line $\{ (v_1, v_2) \mid (1-t) v_1 - t v_2 = c \}$ intersects
$\varphi(\cM_{\bsr})$ in a line segment or a single boundary point.
\end{proof}

For a fixed interior point $\bsr$ in $\mu_{\bT_{U(4)}} (\scO_{\lambda})$,
we identify the symplectic reduction 
$\scM_{\bsr}= \mu_{\bT_{U(4)}}^{-1} (2 \bsr) / \bT_{U(4)}$ with 
the GIT quotient $\Gr(2,4) \GIT \bT_{U(4)}^{\bC}$.
Recall that $\Gr(2,4) \GIT \bT_{U(4)}^{\bC}$ 
is embedded into $\bP^2$ by
\begin{equation}
  \Gr(2,4) \GIT \bT_{U(4)}^{\bC} \hookrightarrow \bP^2,
  \quad
  [p_{ij}] \mapsto 
  [p_{12}p_{34} : p_{13}p_{24}: p_{14}p_{23}],
  \label{eq:embedding_Mr}
\end{equation}
and its image is given by
\begin{equation}
  \{ [\zeta_1:\zeta_2: \zeta_3] \in \bP^2 \mid
  \zeta_2 = \zeta_1 + \zeta_3 \}
  \cong \bP^1.
\end{equation}
Let $D'$ be a divisor in $\Gr(2,4)$ defined by
\begin{equation}
  D' = \{ p_{12} = 0 \} \cup \{p_{23}=0 \} \cup \{ p_{34} = 0 \} ,
\end{equation}
which is contained in an anti-canonical divisor 
\begin{equation}
  D = \{ p_{12} = 0 \} \cup \{p_{23}=0 \} \cup 
    \{ p_{34} = 0 \} \cup \{ p_{14} = 0 \}.
  \label{eq:anti-canonical_Gr(2,4)}
\end{equation}
Recall that the defining equation for $Y \subset \bC^3$ 
given in \eqref{eq:Y} is
\begin{equation}
  x_1 x_2 = 1 + x_3.
\end{equation}
Then $\Gr(2,4) \setminus D'$ is 
identified with $Y \times (\bCx)^2$ by
\begin{equation}
  [p_{12} : p_{13} : p_{14} : p_{23} : p_{24} : p_{34}]
  = [1: x_1: s_1^{-1} s_2 x_3 : s_1: s_2x_2: s_2],
  \label{eq:complement_Gr(2,4)}
\end{equation}
or equivalently,
\begin{equation}
  (x_1, x_2, x_3, s_1, s_2) = \left(
  \frac{p_{13}}{p_{12}}, \frac{p_{24}}{p_{34}}, 
  \frac{p_{14}p_{23}}{p_{12}p_{34}},
  \frac{p_{23}}{p_{12}}, \frac{p_{34}}{p_{12}}
  \right).
\end{equation}
We equip $Y \times (\bCx)^2$ with the symplectic structure
induced from that on $\Gr(2,4)$.
The $\bT_{U(4)}$-action 
\begin{equation}
  \tau \cdot [p_{ij}]_{1 \le i < j \le 4} = [ \tau_i \tau_j p_{ij}],
  \quad
  \tau = \diag (\tau_1, \dots , \tau_4) \in \bT_{U(4)}
\end{equation}
on $\Gr(2,4)$ induces a $\bT_{U(4)}$-action on 
$Y \times (\bCx)^2$ given by
\begin{equation}
 \tau \cdot (x_1, x_2, x_3, s_1, s_2) = 
 \left( \frac{\tau_3}{\tau_2} x_1, \frac{\tau_2}{\tau_3} x_2, x_3, 
         \frac{\tau_3}{\tau_1}s_1, \frac{\tau_3 \tau_4}{\tau_1 \tau_2}s_2 \right).
  \label{eq:torus_action0}
\end{equation}
Then the projection
$Y \times (\bCx)^2 \to \bC_{x_3}$ to the $x_3$-plane
is identified with the restriction to $\Gr(2,4) \setminus D'$
of the GIT quotient of $\Gr(2,4)$
by the $\bT_{U(4)}$-action:
\begin{equation}
  \begin{matrix}
    Y \times (\bCx)^2 & \hookrightarrow
    & \Gr(2,4)^{\mathrm{ss}} \\
    \downarrow & & \downarrow \\
    \bC & \hookrightarrow
    & \Gr(2,4) \GIT \bT_{U(4)}^{\bC}
  \end{matrix},
\end{equation}
where the inclusion 
$\bC \hookrightarrow \Gr(2,4) \GIT \bT_{U(4)}^{\bC}$ 
is given by $x_3 = \zeta_3 / \zeta_1$.

\begin{lemma} \label{lm:max_13}
For each $\bsr$, the bending Hamiltonians
$
 \varphi_{13} \colon \scM_{\bsr} \to \bR
$
and
$
 \varphi_{24} \colon \scM_{\bsr} \to \bR
$
take their maximums
at the points
$\zeta_1 = 0$ and $\zeta_3 = 0$
respectively
under the inclusion 
\eqref{eq:embedding_Mr}.
\end{lemma}

\begin{proof}
Assume $r_1+r_2 \le r_3+r_4$ so that 
$\max \varphi_{13} = r_1 + r_2$.
Take a point 
$[z_i, w_i]_i \in \Gr(2,4) = \Mat_{4 \times 2}(\bC) \GIT_{\lambda} U(2)$ 
such that $\varphi_{13}$ attains its maximum
at $\bsxi = (\nu(z_i, w_i))_{i=1, \dots, 4} \in \scM_{\bsr}$,
where $\nu$ is the Hopf fibration given in \eqref{eq:Hopf}.
Since the side vectors 
$\xi_1 = \nu(z_1, w_1)$ and $\xi_2 = \nu(z_2, w_2)$
of $\bsxi$ have the same direction,
$(z_1, w_1)$ and $(z_2, w_2)$ are proportional,
which implies that
\begin{equation}
  p_{12} = \det \begin{pmatrix}
    z_1 & w_1 \\ z_2 & w_2 \end{pmatrix}
  = 0.
\end{equation}
Hence we obtain $\zeta_1 = p_{12}p_{34} = 0$.
In the case where $r_1+r_2 \ge r_3+r_4$, 
we have $p_{34}=0$ at $[z_i, w_i]_i \in \Gr(2,4)$ such that 
$\max \varphi_{13}$ is attained at 
$\bsxi = (\nu(z_i, w_i))_i \in \scM_{\bsr}$, 
which leads to the same conclusion.
The proof for $\varphi_{24}$ is similar.
\end{proof}

\begin{remark}
The points $\zeta_1 = 0$ and $\zeta_3 = 0$
corresponds to $x_3 = \infty$ and $x_3 = 0$ respectively.
The point $\zeta_2 = p_{13} p_{24} = 0$
corresponding to $x_3 = -1$ is given by
quadrangles
satisfying $\xi_1 \parallel \xi_3$ or $\xi_2 \parallel \xi_4$,
i.e., trapezoids.
If $\bsr$ is general,
then there is only one trapezoid,
which must be planar
and thus lies on the boundary of
\pref{fg:bending_image1}.
\end{remark}

\begin{proposition}
For $t \in (0,1)$, the discriminant locus of 
the completely integrable system $\Psi_t$ 
inside $\Int B_t$ is
\begin{align}
  &\left\{  (u_1, u_2, u_3, u_4) \left| 
  \begin{array}{l}
  \bsr = ( \frac{u_1}2, \frac{u_2}2, \frac{u_3}2,
  \lambda - \frac 12(u_1 + u_2 + u_3)) \in H_2 \\
  u_4 = (1-t) ( r_1 - r_2)^2 - t (r_1 - r_4)^2
  \end{array}
  \right. \right\} \\
  &=
  \{ (2r_1, 2r_2, 2r_3, (1-t) ( r_1 - r_2)^2 - t (r_1 - r_4)^2)
  \mid \bsr = (r_i) \in H_2 \}.
\end{align}
\end{proposition}

\begin{proof}
Let 
\begin{equation}
  L_t(\bsu) = \{ x \mid
  (\psi_{12}(x), \psi_{23}(x), \psi_{34}(x), \psi_t(x))
  = (u_1, u_2, u_3, u_4)  \}
\end{equation}
be a Lagrangian fiber of $\Psi_t$ 
over an interior point 
$\bsu = (u_1, u_2, u_3, u_4)  \in \Int B_t$.
Since the moment map of
the $\bT_{U(4)}$-action on $\scO_{\lambda}$ is given by
\begin{equation}
\mu_{\bT_{U(4)}} = (\psi_{12}, \psi_{23}, \psi_{34}, 
  2\lambda - (\psi_{12} + \psi_{23} + \psi_{34}) ),
\end{equation}
the fiber
$L_t(\bsu)$ lies in the level set $\mu_{\bT_{U(4)}}^{-1}(2 \bsr)$ 
for 
\begin{equation}
  \bsr = (u_1/2, u_2/2, u_3/2, \lambda - ( u_1 + u_2 + u_3 )/2),
\end{equation}
and it is mapped to a level set of $\varphi_t$ under 
the symplectic reduction
\begin{equation}
\pi \colon \mu_{\bT_{U(4)}}^{-1}(2 \bsr) \to \scM_{\bsr}.
\end{equation}
Recall that the stabilizer at $x \in \Gr(2,4)$ of the action of 
the maximal torus $\bT_{SU(4)} = \bT_{U(4)} \cap SU(4)$ 
of $SU(4)$ is nontrivial exactly when 
\begin{equation}
  x \in \{ p_{12} = p_{34} = 0 \} \cup
  \{ p_{13} = p_{24} = 0 \} \cup \{ p_{14} = p_{23} = 0 \}.
\end{equation}
In this case, the corresponding point 
$\bsxi = \pi (x) \in \scM_{\bsr}$,
regarded as a  ``quadrilateral''  in $\bR^3$, 
is contained in a straight line.
If $x $ lies in $\{ p_{12} = p_{34} = 0 \}$
(resp. $\{ p_{14} = p_{23} = 0 \}$),
one has $\bsr \in H_1$
(resp. $\bsr \in H_3$), and
\begin{align}
  \varphi \circ \pi (x) &= ((r_1 + r_2)^2, (r_1 - r_4)^2)
  = (\max_{\scM_{\bsr}} (\varphi_{13})^2, 
  \min_{\scM_{\bsr}} (\varphi_{24})^2) \\
  (\text{resp. }
  \varphi \circ \pi (x) &= ((r_1 - r_2)^2, (r_1 + r_4)^2)
  = (\min_{\scM_{\bsr}} (\varphi_{13})^2, 
  \max_{\scM_{\bsr}} (\varphi_{24})^2) \, )
\end{align}
is the ``lower-right'' (resp. ``upper-left'') node of 
$\partial (\varphi (\scM_{\bsr}))$.
Hence the fiber $L_t( \bsu )$ over
$\bsu = (u_1, u_2, u_3, u_4) = (2r_1, 2r_2, 2r_3, c)$
for $c \in \Int I_{t, \bsr}$ is singular
if and only if the line 
$\{ (v_1, v_2) \mid (1-t) v_1 - t v_2 = c \}$
passes through the ``lower-left'' node 
\begin{equation}
  ( (r_1 - r_2)^2, (r_1 - r_4)^2) 
  = (\min_{\scM_{\bsr}} (\varphi_{13})^2, 
  \min_{\scM_{\bsr}} (\varphi_{24})^2)
\end{equation}
of $\partial(\varphi(\scM_{\bsr}))$ 
as in Figure \ref{fg:singular_fiber},
which means that
$\bsr \in H_2$ and 
\begin{equation}
  c = (1-t) ( r_1 - r_2)^2 - t (r_1 - r_4)^2.
\end{equation}
\begin{figure}[h]
  \centering
    \includegraphics[bb=0 0 79 75]{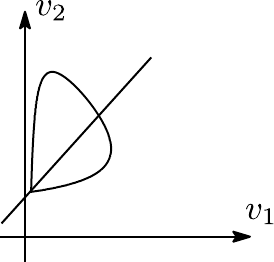}
  \caption{The image in $\partial (\varphi (\scM_{\br}))$ of 
              a singular fiber $L_t(\bsu)$.}
  \label{fg:singular_fiber}
\end{figure}
\end{proof}

Fix $t \in (0,1)$
and let $\bsu = (u_1,u_2,u_3,u_4)$ be a point in $\Int B_t$
which is not on the discriminant of $\Psi_t$,
so that
$
 L_t(\bsu) \coloneqq \Psi_t^{-1}(\bsu)
$
is a Lagrangian torus fiber.
First we assume
\begin{equation} \label{eq:1324}
\varphi_{13}(\bsu) / \varphi_{24}(\bsu) \gg 1,
\end{equation} 
so that
$\pi (L_t(\bsu))$ is a small simple closed curve enclosing the point 
$\zeta_{\max} \in \scM_{\bsr} \cong \bP^1$ at which 
$\max \varphi_t$ is attained.
One can deform $\pi(L_t(\bsu))$ into a level set of $\varphi_{13}$ 
without crossing $x_3 = 0, -1$
by a Hamiltonian flow which sends $\zeta_{\max}$ to $x_3 = \infty$.
We may assume that the Hamiltonian is supported in a small neighborhood 
of the domain bounded by $\pi (L_t(\bsu))$.
Note that if $\bsr$ lies in the wall $H_1$, then one has
$\zeta_{\max} = [0:1:1]$, which corresponds to $x_3 = \infty$.
Hence we may assume in this case that 
the support of the Hamiltonian 
does not contain the point $x_3 = \infty$.
We pull-back the Hamiltonian,
and multiply a cut-off function supported near the level set
$\mu_{\bT_{U(4)}}^{-1}(\bsr)$.
Since it Poisson-commutes with the Hamiltonians of the 
$\bT_{U(4)}$-action near $\mu_{\bT_{U(4)}}^{-1}(\bsr)$,
the induced Hamiltonian flow on $\Gr(2,4)$
sends $L_t(\bsu)$ to a Lagrangian torus fiber $L_{\Gamma}(\bsu')$ of 
the integrable system $\Psi_{\Gamma}$ over 
some point
$\bsu' = (u_{12},u_{23},u_{34},u_{13}) \in \Int \Delta_{\Gamma}$.
We simultaneously deform the locus
$\bigcup_{\bsr} \pi^{-1}( \{ \varphi_t = \max_{\scM_{\bsr}} \varphi_t \})$
to
\begin{equation}
  \{ \varphi_{13} = \max \varphi_{13} \} 
  = \{ p_{12} p_{34} = 0 \} \subset D,
\end{equation}
so the Lagrangian torus fiber does not cross this locus
during the Hamiltonian isotopy
from $L_t(\bsu)$ to $L_{\Gamma}(\bsu')$.
Note that $\Psi_t$ does not have a global action--angle coordinate
because of the existence of the discriminant in $\Int B_t$,
and $\bsu'$ is a local action--angle coordinate
in the region \pref{eq:1324}.

One can further deform $\pi (L_{\Gamma}(\bsu'))$ to 
a circle $\{ x_3 \in \bC \mid |x_3| = R \}$
by a Hamiltonian flow
without crossing $x_3 = 0, -1, \infty$.
The pull-back of the Hamiltonian gives a Hamiltonian flow $\Gr(2,4)$
which sends $L_{\Gamma}(\bsu')$ to a Lagrangian torus of the form
$T_{r,R} \times T'$ for a Clifford type Lagrangian torus 
$T_{r,R} \subset Y$
defined in \eqref{eq:T_rR}
and a two-torus $T' \subset (\bCx)^2$.

Note that $L_{\Gamma}(\bsu')$ (and hence $T_{r, R} \times T'$) 
can be deformed into a toric fiber in the toric variety $X_{\Gamma}$.
The anti-canonical divisor $D$ can be deformed into 
the toric divisor in $X_{\Gamma}$,
and thus the Maslov index of a holomorphic disk
in $(\Gr(2,4), L_{\Gamma}(\bsu'))$ 
(and hence that in $(\Gr(2,4), T_{r,R} \times T')$)
is twice the intersection number with $D$
(if the Lagrangian torus does not intersect $D$).
Then, by considering the projection $Y \times (\bCx)^2 \to \bC$,
no Lagrangian torus bounds holomorphic disks of 
non-positive Maslov index through the Lagrangian isotopy
from $L_t(\bsu)$ to $T_{r, R} \times T'$,
where we consider a compatible almost complex structure which induces 
the standard complex structure on $\bC$ under 
the symplectic reduction.
We also note that each holomorphic disk  in $\Gr(2,4) \setminus D'$
of Maslov index two bounded by $L_t(\bsu)$ 
intersects the divisor $\{ p_{14} = 0\} \subset D$ at one point, 
and hence it descends to a disk in $Y$ of the same Maslov index.
For $\beta \in \pi_2(Y, T_{r,R})$, 
let $\tilde{\beta} \in \pi_2(\Gr(2,4), L_t(\bsu))$ denote its lift 
induced from an inclusion 
$Y \cong Y \times \{ \mathrm{pt} \} \subset Y \times (\bCx)^2$.


\begin{lemma} \label{lm:beta12}
The boundaries
$\partial \tilde\beta_1, \partial \tilde\beta_2 \in \pi_1(L_t(\bsu))$ 
of the lifts of $\beta_1, \beta_2 \in \pi_2(Y, T_{r,R})$ 
defined in Example \ref{eg:Auroux}
are represented by Hamiltonian $S^1$-orbits of
$\psi_{13} - \psi_{23}$ and
$\psi_{13} - \psi_{12} - \psi_{23} - \psi_{34}$, respectively.
The symplectic areas
of $\tilde\beta_1, \tilde\beta_2$ are given by
\begin{align}
  \omega (\tilde\beta_1) 
  & = u_{12} - (u_{12} + u_{23} - u_{13}) = u_{13} - u_{23}, 
  \label{eq:symp_area1} \\
  \omega (\tilde\beta_2) 
  & = u_{13} - (u_{12} + u_{23} + u_{34}) + \lambda
  \label{eq:symp_area2}.
\end{align}
\end{lemma}
 
\begin{proof}
Theorem \ref{th:toric_deg1} allows us to  
consider holomorphic disks in the central fiber $X_{\Gamma}$ of the toric degeneration
\begin{equation}
 \begin{array}{cll}
  \frakX_{\Gamma} & \!\!\! = \{ (\bsp, t) 
  \in \bP (\textstyle{\bigwedge^2 \bC^4}) \times \bC
  \mid p_{13}p_{24} = t p_{12}p_{34} + p_{14}p_{23} \} \\
  \downarrow \\
  \bC
 \end{array}
  \label{eq:toric_deg(2,4)}
\end{equation}
associated with $\Gamma$,
instead of those in $\Gr(2,4)$.
We recall a construction of the family 
$\Psi_{\Gamma}^t = (\psi_{ij}^t)_{\edg (i,j) \in \Prn \Gamma}$ 
of completely integrable systems connecting $\Psi_{\Gamma}$ 
and the toric moment map. 
Extend the actions of  $\bT_{U(4)}$ and $G(1,3) \subset U(4)$ 
on $\Gr(2,4)$ 
to those on $\bP(\bigwedge^2 \bC^4)$
in an obvious way,
and let 
\begin{align}
  \tilde\mu_{\bT_{U(4)}} (\bsp)
  & = (\tilde\psi_{12}(\bsp) , \tilde\psi_{23}(\bsp), 
         \tilde\psi_{34}(\bsp), \tilde\psi_{14}(\bsp)) \\
  & = \frac{\lambda}{2\sum |p_{ij}|^2} 
        \left( \sum_{j \ne 1} |p_{1j}|^2, \sum_{j \ne 2} |p_{2j}|^2, 
        \sum_{j \ne 3} |p_{3j}|^2 , \sum_{j \ne 4} |p_{4j}|^2
        \right),  \\
  \tilde\mu_{G(1,3)}(\bsp) 
  &= \frac{\lambda}{2 \sum |p_{ij}|^2}
  \begin{pmatrix} 
    \sum_{j \ne 1} |p_{1j}|^2 
      & p_{23} \overline{p_{13}} + p_{24} \overline{p_{14}} \\
    p_{13} \overline{p_{23}} + p_{14} \overline{p_{24}}
      & \sum_{j \ne 2} |p_{2j}|^2
  \end{pmatrix}
\end{align}
denote the moment  maps of these actions.
Then we obtain an extension
$\tilde\psi_{13} \colon \bP(\bigwedge^2 \bC^4) \to \bR$
of $\psi_{13}$,
which associates to 
$\bsp$
the maximum eigenvalues of $\tilde\mu_{G(1,3)}(\bsp)$.
The family of integrable systems $\Psi_{\Gamma}^t$ is given by 
 the restrictions $\psi_{ij}^t = \tilde\psi_{ij}|_{X_t}$,
$\edg(i,j) \in \Prn \Gamma$ 
to each fiber $X_t = f_{\Gamma}^{-1}(t)$.
Note that Poisson commutativity of $\Psi_{\Gamma}^t$
follows from the fact that 
the actions of  $\bT_{U(4)}$ and 
$G(1,3) \subset U(4)$ on $\bP (\bigwedge^2 \bC^2)$
preserve each fiber $X_t$.

Using the defining equation
\begin{equation}
  p_{13}p_{24} = p_{14}p_{23}
\end{equation}
of the central fiber $X_{\Gamma}$,
we have
\begin{align}
  \psi_{13}^0 
  &= \frac{\lambda}{2 \sum |p_{ij}|^2} ( |p_{12}|^2 + |p_{13}|^2 
    + |p_{14}|^2 + |p_{23}|^2 + |p_{24}|^2) \\
  &= \frac{\lambda}{2} \left( 1 - \frac{|p_{34}|^2}{\sum |p_{ij}|^2} \right),
\end{align}
which implies that Hamiltonian $S^1$-action of 
$\psi^0_{13}$ is given by 
\begin{equation}
  e^{\sqrt{-1} \theta} \cdot \bsp = 
  [p_{12}: p_{13}: p_{14}: p_{23}: p_{24}: e^{-\sqrt{-1}\theta} p_{34}].
\end{equation}
We consider a deformation family
\begin{equation}
  \frakY_{\Gamma} = \{ (x_1,x_2,x_3, t) \in \bC^3 \times \bC \mid 
  x_1 x_2 = t + x_3 \}
  \to \bC_t
\end{equation}
of $Y$ induced from the toric degeneration \eqref{eq:toric_deg(2,4)},
whose central fiber is given by
\begin{equation}
  Y_{\Gamma} = \{ (x_1, x_2, x_3) \in \bC^3 \mid x_1 x_2 = x_3 \}.
\end{equation}
The complement $X_{\Gamma} \setminus D'_{\Gamma}$ 
of the divisor
\begin{equation}
  D'_{\Gamma} = \{ p_{12} = 0 \} \cup \{p_{23}=0 \} \cup \{ p_{34} = 0 \}
  \label{eq:toric_divisor}
\end{equation}
on $X_{\Gamma}$ is identified with
$Y_{\Gamma} \times (\bCx)^2$,
on which the Hamiltonian $S^1$-action of $\psi^0_{13}$
is given by 
\begin{equation}
  e^{\sqrt{-1}\theta} (x_1, x_2, x_3, s_1,s_2) = 
  (x_1, e^{\sqrt{-1}\theta}x_2, e^{\sqrt{-1}\theta}x_3, s_1, 
  e^{- \sqrt{-1}\theta} s_2).
  \label{eq:torus_action1}
\end{equation}
Since the lifts
 $\tilde\beta_1$, $\tilde\beta_2$,
regarded as relative homotopy classes in 
$Y_{\Gamma} \times (\bCx)^2$, 
are represented by holomorphic disks of the form 
$(x_2, s_1, s_2)=\text{const.}$ and 
$(x_1, s_1, s_2)=\text{const.}$, respectively,
it follows from \eqref{eq:torus_action0} and \eqref{eq:torus_action1}
that
$\partial \tilde\beta_1$, $\partial \tilde\beta_2$ are represented by
the following Hamiltonian $S^1$-orbits
\begin{align}
  (e^{\sqrt{-1}\theta} x_1, x_2, e^{\sqrt{-1}\theta}x_3, s_1, s_2) 
  & \longleftrightarrow \psi_{13}^0 - \psi_{23}^0, 
    \label{eq:bdry_beta1}\\
  (x_1, e^{\sqrt{-1}\theta} x_2, e^{\sqrt{-1}\theta}x_3, s_1, s_2) 
  & \longleftrightarrow  
     \psi_{13}^0 + \psi_{14}^0 \notag \\
  & \phantom{ \longleftrightarrow }   
  = \psi_{13}^0 - \psi_{12}^0 -\psi_{23}^0 -\psi_{34}^0 + \text{const.},
    \label{eq:bdry_beta2}
\end{align} 
respectively.
Here we recall the formula \cite[Theorem 8.1]{MR2282365}
for symplectic area of holomorphic disks Maslov index 2 in 
a toric manifold.
Suppose that a holomorphic disk
$w \colon (D^2, \partial D^2) \to (X_{\Gamma}, L_0(\bsu))$
of Maslov index 2 
intersects a toric divisor $D_{\bsv}$ corresponding to
the facet
\begin{equation}
  \{ \bsu \in \Delta_{\Gamma} \mid 
  \ell(\bsu) = \langle \bsv, \bsu \rangle - \tau = 0 \}
\end{equation}
of the moment polytope.
Then 
\begin{equation}
  [ w(\partial D^2)] = \bsv \in H^1(L_0(\bsu); \bZ) \cong \bZ^2,
\end{equation}
and the symplectic ares of $w$ is given by
\begin{equation}
  \int_{D^2} w^* \omega = \ell (\bsu).
  \label{eq:area_formula}
\end{equation}
Comparing \eqref{eq:bdry_beta1} and \eqref{eq:bdry_beta2}
with the defining inequalities
\begin{equation}
\begin{alignedat}{9}
  \ldbox{\lambda} &&&& \ldbox{u_{12}+u_{23}+u_{34}-\lambda} &&&& \ldbox{0} \\
  & \uge && \dge && \uge && \dge \\
  && \ldbox{u_{13}} &&&& u_{12}+u_{23}-u_{13} \\
  &&& \uge && \dge && \\
  &&&& \ldbox{u_{12}} 
\end{alignedat}
\end{equation}
of the moment polytope $\Delta_{\Gamma}$,
it follows that $\tilde\beta_1$ and $\tilde\beta_2$ intersect 
toric divisors corresponding to the facets of $\Delta_{\Gamma}$
defined by
\begin{align}
  \ell_1 (\bsu) &= u_{12} - (u_{12} + u_{23} - u_{13}) = u_{13} - u_{23}, 
    \label{eq:disk1} \\
  \ell_2(\bsu)  & = u_{13} - (u_{12} + u_{23} + u_{34} - \lambda),
    \label{eq:disk2}
\end{align}
respectively.
By topological reason, the results 
$\omega (\tilde\beta_i) = \ell_i (\bsu)$ 
proved in $(X_{\Gamma}, L_0(\bsu))$ 
is true also in $(X_t, L_t(\bsu))$ for $t > 0$.
\end{proof}

Note that the defining functions \eqref{eq:disk1} and \eqref{eq:disk2}
for $\Delta_{\Gamma}$ correspond to the following triangle inequalities
\begin{align}
  u(2,3) - u(1,2) &\le u(1,3), \\
  u(3,4) - u(1,4) &\le u(1,3)
\end{align}
in the coordinates defined in \eqref{eq:bending_coord},
respectively.

Next we assume that  $\bsu \in \Int B_t$ satisfies
$
 \varphi_{13}(\bsu) / \varphi_{24}(\bsu) \ll 1.
$
Then the image $\pi (L_t(\bsu)) \subset \scM_{\bsr}$ can be 
deformed into a level set of $\varphi_{24}$,
which is a simple closed curve enclosing the point $\zeta_3 = 0$,
and thus $L_t(\bsu)$ is deformed into a fiber $\Psi_{\Gamma'}(\bsu'')$
of the other completely integrable system $\Psi_{\Gamma}$
for some $\bsu'' = (u_{12},u_{23},u_{34},u_{24})$.
Since the point $\zeta_3=0$ corresponds to the origin in the $x_3$-plane,
the fiber $L_{\Gamma'}(\bsu'')$  
can be deformed into $T_{r,R} \times T'$
for a Chekanov type Lagrangian torus $T_{r,R} \subset Y$ and 
a two-torus $T'$ in $(\bCx)^2$.

\begin{lemma} \label{lm:beta3}
The boundary $\partial \tilde\beta_3 \in \pi_1(L_t(\bsu))$ 
of the lift of the class $\beta_3 \in \pi_2(Y, T_{r,R})$
defined in Example \ref{eg:Auroux}
is represented by a Hamiltonian $S^1$-orbit of
$\psi_{24}$, and the symplectic area
of $\tilde\beta_3$ is given by
\begin{equation}
  \omega (\tilde\beta_3) = \lambda - u_{24}.
\end{equation}
\end{lemma}

\begin{proof}
We consider the central fiber $X_{\Gamma'}$ of the toric degeneration 
\begin{equation}
  \frakX_{\Gamma'} =\{ (\bsp, t) 
  \in \bP (\textstyle{\bigwedge^2 \bC^4}) \times \bC
  \mid p_{13}p_{24} =  p_{12}p_{34} + t p_{14}p_{23} \}
  \label{eq:toric_deg(2,4)'}
\end{equation}
associated with $\Gamma'$, and 
let $\Psi_{\Gamma'}^0 = (\varphi_{ij}^0)_{\edg(i,j) \in \Prn \Gamma'}$ 
denote the toric moment map on $X_{\Gamma'}$
throughout this proof.
Since $X_{\Gamma'}$ is defined by
\begin{equation}
  p_{13}p_{24} =  p_{12}p_{34} ,
\end{equation}
one has
\begin{align}
  \psi_{24}^0 
  &= \frac{\lambda}{2 \sum |p_{ij}|^2} (|p_{12}|^2+|p_{13}|^2
    +|p_{23}|^2 + |p_{24}|^2+|p_{34}|^2) \\
  &= \frac{\lambda}{2} \left( 1 -  \frac{|p_{14}|^2}{\sum |p_{ij}|^2} \right),
\end{align}
which implies that its Hamiltonian $S^1$-action is 
\begin{equation}
  e^{\sqrt{-1} \theta} \cdot \bsp = 
  [p_{12}: p_{13}: e^{-\sqrt{-1}\theta} p_{14}: p_{23}: p_{24}: p_{34}].
\end{equation}
Consider the family 
\begin{equation}
  \frakY_{\Gamma'} = \{ (x_1,x_2,x_3, t) \in \bC^3 \times \bC \mid 
  x_1 x_2 = 1 + t x_3 \}
\end{equation}
of affine varieties induced from \eqref{eq:toric_deg(2,4)'},
whose central fiber is given by 
\begin{equation}
  Y_{\Gamma'} = \{ (x_1, x_2, x_3) \in \bC^3 \mid x_1 x_2 = 1 \}.
\end{equation}
For a  divisor $D'_{\Gamma'}$ on $X_{\Gamma'}$ 
defined by the same equation as \eqref{eq:toric_divisor}, 
the complement $X_{\Gamma'} \setminus D'_{\Gamma'}$
is identified with $Y_{\Gamma'} \times (\bCx)^2$,
on which the Hamiltonian $S^1$-action 
of $\psi_{24}^0$ is given by
\begin{equation}
  e^{\sqrt{-1}\theta} (x_1, x_2, x_3, s_1,s_2) = 
  (x_1, x_2, e^{-\sqrt{-1}\theta}x_3, s_1, s_2).
\end{equation}
From this and \eqref{eq:torus_action0}, 
the boundary $\partial \tilde\beta_3$,
regarded as a class in the central fiber,
is represented by 
a $S^1$-orbit of $\psi_{24}$.
Since the moment polytope
$\Delta_{\Gamma'}$ is defined by
\begin{equation}
\begin{alignedat}{9}
  \ldbox{\lambda} &&&& \ldbox{u_{12}+u_{23}+ u_{34}-\lambda} &&&& \ldbox{0} \\
  & \uge && \dge && \uge && \dge \\
  && \ldbox{u_{24}} &&&& u_{23}+u_{34}-u_{24} \\
  &&& \uge && \dge && \\
  &&&& \ldbox{u_{23}} 
\end{alignedat},
\end{equation}
the holomorphic disk in $\tilde\beta_3$ intersects the toric divisor
corresponding to the facet of $\Delta_{\Gamma'}$ defined by  
\begin{equation}
  \ell_3(\bsu) = \lambda - u_{24},
\end{equation}
which corresponds to the triangle inequality
\begin{equation}
  u(2,4) \le u(1,2) + u(1,4).
\end{equation}
Lemma \ref{lm:beta3} follows the area formula \eqref{eq:area_formula}
and invariance of symplectic areas under the deformation.
\end{proof}

Finally we take a point $\bsu \in B_t$ on the wall,
and consider the lift $\tilde{\alpha}$
of the class
$\alpha \in \pi_2(Y, T_{r,R})$ of Maslov index zero.
Since $\alpha = \beta_1 - \beta_2$, 
we have the following:

\begin{lemma} \label{lm:alpha}
The boundary
$\partial \tilde\alpha \in \pi_1(L_t(\bsu))$
of the lift of $\alpha$  
is represented by a Hamiltonian $S^1$-orbit of
$\psi_{12}+ \psi_{34}$, and the symplectic area
of $\tilde\alpha$ is given by
\begin{equation}
  \omega (\tilde\alpha) = u_{12} + u_{34} - \lambda.
\end{equation}
\end{lemma}

One can see this also from the fact that 
the class $\tilde\alpha$ is represented by
a disk of the form $(x_2, s_1, s_2) = \text{const.}$, 
and thus 
\eqref{eq:torus_action0} implies that 
the boundary of the disk is a Hamiltonian $S^1$-orbit 
of $\psi_{12}+\psi_{34}$.

From Lemmas \ref{lm:beta12}, \ref{lm:beta3}, \ref{lm:alpha},
the  functions 
$z_{\beta}(b) = T^{\omega (\beta)} \hol_{b}(\partial \beta)$
for $\beta = \tilde\beta_1, \tilde\beta_2, \tilde\beta_3, \tilde\alpha$ 
are given by
\begin{align}
  z_{\tilde\beta_1} &= \frac{y_{13}}{y_{23}}, \\
  z_{\tilde\beta_2} &= \frac{q y_{13}}{y_{12}y_{23}y_{34}}, \\
  z_{\tilde\beta_3} &= \frac{q}{y_{24}},\\
  z_{\tilde\alpha} &= \frac{y_{12}y_{34}}{q},
\end{align}
where $q = T^{\lambda}$ for the Novikov parameter $T$,
and therefore the coordinate change \eqref{eq:cluster_transf} 
gives 
\begin{align}
  z_{\tilde\beta_3} = z_{\tilde\beta_1} + z_{\tilde\beta_2}
  = z_{\tilde\beta_2}(1 + z_{\tilde\alpha}),
  \label{eq:wallcrossing_Gr}
\end{align}
which coincides with the wall crossing formula
\eqref{eq:Auroux_wallcrossing}.

\begin{remark}
The Lagrangian torus fibers
$L_{\Gamma}(\bsu_0)$ and $L_{\Gamma'}(\bsu'_0)$ 
above the centers
\begin{align}
  \bsu_0 & = ( u_{12}, u_{23}, u_{34}, u_{13}) 
  = \biggl( \frac{\lambda}2, \frac{\lambda}2, 
       \frac{\lambda}2, \frac{3\lambda}4 \biggr)
  \in \Delta_{\Gamma}, \\
  \bsu'_0 & =( u_{12}, u_{23}, u_{34}, u_{24}) 
  = \biggl(\frac{\lambda}2, \frac{\lambda}2, 
       \frac{\lambda}2, \frac{3\lambda}4 \biggr)
  \in \Delta_{\Gamma'}
\end{align}
are monotone
by \cite[Theorem B]{1801.07554}.
The fiber $L_{\Gamma}(\bsu_0)$ 
(resp. $L_{\Gamma'}(\bsu'_0)$) is contained in the
complement $\Gr(2,4) \setminus D$ of 
the anti-canonical divisor $D$ given by
\eqref{eq:anti-canonical_Gr(2,4)}, 
and the image under the projection
$\pi \colon \Gr(2,4) \setminus D \to \bC_{x_3}$
is a simple closed curve enclosing 
$x_3 = 0, -1$ (resp. $x_3 = 0$).
Note that 
$L_{\Gamma}(\bsu_0)$ and $L_{\Gamma'}(\bsu'_0)$ lie 
on the level set $\mu_{\bT_{U(4)}}^{-1}(\lambda/2, \dots, \lambda/2)$,
and the intersection 
$\pi (L_{\Gamma}(\bsu_0)) \cap \pi (L_{\Gamma'}(\bsu'_0))$,
viewed as a subset in the polygon space, consists of two points
corresponding to spatial quadrilaterals such that the configuration
of vertices define a regular tetrahedron.
The fiber $L = L_{\Gamma} (\lambda/2, \dots, \lambda/2)$ of $\Psi_{\Gamma}$
on the boundary point 
\begin{equation}
  (u_{12}, u_{23}, u_{34}, u_{13}) 
  = (\lambda/2, \lambda/2, \lambda/2, \lambda/2)
  \in \partial \Delta_{\Gamma}
\end{equation}
is a Lagrangian $U(2) \cong S^1 \times S^3$
(see \cite[Proposition 2.7]{MR3601889}). 
It is easy to see that $\varphi_{13} =0$ on $L$
and the image $\varphi(L)$ is a line segment connecting
$x_3 = 0$ and $x_3 = -1$ (see Figure \ref{fg:mutation_config}). 
Therefore the Lagrangian torus $L_{\Gamma'}(\bsu'_0)$ together with
the inverse image in $L$ of the line segment connecting 
$x_3 = -1$ and the intersection point 
$\varphi(L_{\Gamma'}(\bsu'_0)) \cap \varphi (L)$
give a higher dimensional mutation configuration 
discussed in \cite[Section 5]{pascaleff2017wall}.
Similarly, the fiber $L'$ of $\Psi_{\Gamma'}$ on
the boundary point
\begin{equation}
  (u_{12}, u_{23}, u_{34}, u_{24}) 
  = (\lambda/2, \lambda/2, \lambda/2, \lambda/2)
  \in \partial \Delta_{\Gamma'}
\end{equation}
is a Lagrangian $U(2)$, 
whose image in the $x_3$-plane is $\varphi (L') = (- \infty, -1]$. 
Thus the pair of $L_{\Gamma}(\bsu_0)$ and 
the inverse image of the line segment connecting $x_3=-1$ and 
$\varphi(L_{\Gamma}(\bsu_0)) \cap \varphi (L')$ gives 
a mutation configuration.
\begin{figure}[h]
  \centering
    \includegraphics[bb=0 0 152 73]{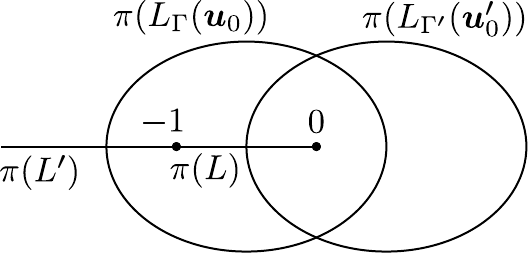}
  \caption{Images of Lagrangian fibers in the $x_3$-plane.}
  \label{fg:mutation_config}
\end{figure}
The Lagrangian tori
$L_{\Gamma}(\bsu_0)$ and $L_{\Gamma'}(\bsu'_0)$
are related by
(the multiplication with a 2-torus of)
a Lagrangian surgery
studied in \cite{MR3701933},
and the wall-crossing formula between them
is obtained from
\cite[Theorem 5.7]{pascaleff2017wall}.
\end{remark}


\section{Wall-crossing formula on general $\Gr(2,n)$}

In this section we complete the proof of Theorem \ref{th:wall-crossing}.

We consider the family
$f_{\Gamma''} \colon \frakX_{\Gamma''} \to \bC^{n-4}$
associated with the subdivision $\Gamma''$
given by common diagonals in $\Gamma$ and $\Gamma'$,
and let 
\begin{equation}
  \Psi^0_{\Gamma} = ((\psi^0_{ij})_{\edg(i,j) \in \Prn \Gamma''}, \psi^0_{ac}),
  \quad
  \Psi^0_{\Gamma'} = ((\psi^0_{ij})_{\edg(i,j) \in \Prn \Gamma''}, \psi^0_{bd})
\end{equation}
be the completely integrable systems on the central fiber 
$X_0 = f_{\Gamma''}^{-1}(\bszero)$
obtained by deforming $\Psi_{\Gamma}$ and $\Psi_{\Gamma'}$,
respectively.
Then $\Psi_t$ is deformed into the
completely integrable system 
\begin{equation}
  \Psi_t^0 = ((\psi_{ij}^0)_{\edg(i,j) \in \Prn \Gamma''}, 
  (1-t) (\varphi_{ac}^0)^2 - t (\varphi_{bd}^0)^2 )
\end{equation}
on $X_0$,
where we set 
$\varphi_{ij}^0 = \psi_{ij}^0 - \frac12 \sum_{k=i}^{j-1} \psi_{k, k+1}^0$,
and thus Corollary \ref{cr:toric_deg2} implies that
we may work on $X_0$.
For an open dense subset $Y^{\circ}$ of $Y$ defined by
\begin{equation}
  Y^{\circ} = \{ (x_1, x_2, x_3) \in Y \mid x_3 \ne 0 \}
  \cong \bC^2 \setminus \{ x_1 x_2 = 1 \},
\end{equation}
the isomorphisms given in Corollary \ref{cr:complement} 
and \eqref{eq:complement_Gr(2,4)}
yields
\begin{equation}
  X_0 \setminus D_0 \cong
  Y^{\circ} \times (\bCx)^2 
  \times \bT_{\Gamma'' \setminus \Gamma_0}^{\bC}
  \cong Y^{\circ} \times (\bCx)^{2n-6}
\end{equation}
such that the restriction to  $X_0 \setminus D_0$ 
of the GIT quotient 
$X_{0}^{\mathrm{ss}} \to X_0 \GIT \bT_{\Gamma''}^{\bC}
\cong \bP^1$ 
is identified with the projection
\begin{equation}
  f \colon Y^{\circ} \times (\bCx)^{2n-6} \to Y^{\circ} 
  \to \bC_{x_3}
\end{equation}
to the $x_3$-plane.
Since $\bT_{\Gamma''}^{\bC}$ 
is the complexification of a torus $\bT_{\Gamma''}$ 
generated by Hamiltonian flows of
$(\psi_{ij}^0)_{\edg(i,j) \in \Gamma''}$,
each Lagrangian torus fiber $L^0_t(\bsu)$ of $\Psi^0_t$ 
is mapped by $f$ to a level set of $\varphi_t$
in a complex 1-dimensional polygon space
$X_0 \GIT \bT_{\Gamma''}^{\bC} \cong \Gr(2,4) \GIT \bT_{U(4)}^{\bC}$,
which implies that $L^0_t(\bsu)$ can be deformed into a 
Lagrangian torus of the form 
$T_{r,R} \times T'$ for some $(r,R) \in \bR \times \bR_{>0}$
and a $(2n-6)$-torus $T'$ in $(\bCx)^{2n-6}$. 
We first assume that $T_{r,R}$ is of Clifford type.
From the argument in the previous section, 
the lifts $\tilde\beta_1, \tilde\beta_2 \in \pi_2(X_0, L^0_t(\bsu))$ of
classes $\beta_1$, $\beta_2$ in $Y$ have symplectic areas
\begin{align}
  \ell_1(\bsu) &= u_{ac} + u_{ab} - u_{bc} - \sum_{i=a}^{b-1} u_{i,i+1},\\
  \ell_2(\bsu) &= u_{ac} + u_{ad} - u_{cd} - \sum_{i=a}^{c-1} u_{i,i+1},
\end{align}
which correspond to the triangle inequalities
\begin{align}
  u(b,c) - u(a, b) &\le u(a, c), \\
  u(c,d) - u(a, d) &\le u(a, c),
\end{align}
respectively. 
On the other hand, in the case where $T_{r,R}$ is of Chekanov type,
the symplectic area of $\tilde\beta_3$ is by
\begin{equation}
  \ell_3(\bsu) = -u_{bd} + u_{ad} + u_{cd} - \sum_{i=a}^{b-1} u_{i,i+1},
\end{equation}
which corresponds to
\begin{equation}
  u(b, d) \le u(c, d) + u(a, d).
\end{equation}
Hence the functions 
$z_{\tilde\beta_i}$ are given by
\begin{align}
  z_{\tilde\beta_1} 
    &= \frac{y_{ab}y_{ac}}{y_{bc} \prod_{i=a}^{b-1} y_{i,i+1}}, \\
  z_{\tilde\beta_2} 
    &= \frac{y_{ad}y_{ac}}{y_{cd} \prod_{i=a}^{c-1} y_{i,i+1}}, \\
  z_{\tilde\beta_3} 
    &= \frac{y_{ab}y_{ad}}{y_{bd} \prod_{i=a}^{b-1} y_{i,i+1}}.
\end{align}
Since $z_{\tilde\alpha}$ corresponding to 
$\alpha = \beta_1 - \beta_2$ is given by
\begin{equation}
  z_{\tilde\alpha} = \frac{z_{\tilde\beta_1}}{z_{\tilde\beta_2}}
  = \frac{y_{ab}y_{cd} \prod_{i=b}^{c-1} y_{i,i+1}}{y_{ad}y_{bc}} ,
\end{equation}
the coordinate change \eqref{eq:geom-lift} is equivalent to
the wall-crossing formula \eqref{eq:wallcrossing_Gr},
which complete the proof.

\bibliographystyle{amsalpha}
\bibliography{bibs}


\noindent
Yuichi Nohara \\
Department of Mathematics, 
School of Science and Technology, \\
Meiji University\\
1-1-1 Higashi-Mita, Tama-ku, Kawasaki-shi, 
Kanagawa 214-8571 , Japan\\
{\em e-mail address}\ : \  nohara@meiji.ac.jp
\ \vspace{0mm} \\

\noindent
Kazushi Ueda \\
Graduate School of Mathematical Sciences, 
The University of Tokyo,\\
3-8-1 Komaba,
Meguro-ku,
Tokyo,
153-8914,
Japan\\
{\em e-mail address}\ : \  kazushi@ms.u-tokyo.ac.jp
\ \vspace{0mm} \\

\end{document}